\documentclass[oneside,english]{amsart}
\usepackage[T1]{fontenc}
\usepackage[latin9]{inputenc}
\usepackage{xcolor}
\usepackage{amstext}
\usepackage{amsthm}
\usepackage{amssymb}
\PassOptionsToPackage{normalem}{ulem}
\usepackage{ulem}

\makeatletter

\providecolor{lyxadded}{rgb}{0,0,1}
\providecolor{lyxdeleted}{rgb}{1,0,0}

\DeclareRobustCommand{\lyxsout}[1]{\ifx\\#1\else\sout{#1}\fi}

\numberwithin{equation}{section}
\numberwithin{figure}{section}
\theoremstyle{plain}
\newtheorem{thm}{\protect\theoremname}[section]
\theoremstyle{remark}
\newtheorem{rem}[thm]{\protect\remarkname}
\theoremstyle{plain}
\newtheorem{conjecture}[thm]{\protect\conjecturename}
\theoremstyle{definition}
\newtheorem{defn}[thm]{\protect\definitionname}
\theoremstyle{plain}
\newtheorem{lem}[thm]{\protect\lemmaname}
\theoremstyle{plain}
\newtheorem{cor}[thm]{\protect\corollaryname}
\theoremstyle{plain}
\newtheorem{prop}[thm]{\protect\propositionname}

\makeatother

\usepackage{babel}
\providecommand{\conjecturename}{Conjecture}
\providecommand{\corollaryname}{Corollary}
\providecommand{\definitionname}{Definition}
\providecommand{\lemmaname}{Lemma}
\providecommand{\propositionname}{Proposition}
\providecommand{\remarkname}{Remark}
\providecommand{\theoremname}{Theorem}

\begin{document}
\newcommand{\MA}{{Monge-Amp{\`e}re}\ } 
\title{A Liouville's theorem for some Monge-Amp{\`e}re type equations}
\author{Hao Fang}
\address{14 MacLean Hall, Department of Mathematics, University of Iowa, Iowa
City, IA, 52242}
\email{hao-fang@uiowa.edu}
\author{Biao Ma}
\address{14 MacLean Hall, Department of Mathematics, University of Iowa, Iowa
City, IA, 52242}
\email{biao-ma@uiowa.edu}
\author{Wei Wei}
\address{Department of Mathematics, Nanjing University, Nanjing, P.R.China,
210093}
\email{wei\_wei@nju.edu.cn}
\begin{abstract}
We study a \MA type equation that interpolates the classical $\sigma_{2}$-Yamabe
equation in conformal geometry and the 2-Hessian equation in dimension
4. 
\end{abstract}

\maketitle

\section{Introduction}

In this paper, we study some \MA type equation on $\mathbb{R}^{4}$.
Let $u(x)\in C^{2}(\mathbb{R}^{4})$, $\rho\in\mathbb{R}$, define
the following symmetric 2 tensor
\[
A(\rho,u)=-u_{ij}+\rho u_{i}u_{j}-\frac{\rho}{2}|\nabla u|^{2}\delta_{ij}\in{\rm Sym^{2}(\mathbb{R}^{4})},
\]
where we use $\nabla_{i}u=u_{i}$ and $\nabla_{ij}u=u_{ij}$ to denote
partial derivatives of $u$ with respect to the coordinate system
$\{x^{i}\}$ of $\mathbb{R}^{4}.$ Note that when $\rho=0,$ $A_{\rho}$
is just the Hessian matrix of the function $u$. While $\rho=1,$
up to a multiplying factor, $A_{\rho}$ is the Schouten curvature
of the metric $\exp(2u)g_{E}$. We define $\sigma_{k}(A_{\rho})$
to be the degree $k$ symmetric polynomial of all eigenvalues of $A_{\rho}.$ 

In this paper, we consider fully non-linear problems
\begin{equation}
\sigma_{2}(A(\rho,u(x)))=f(u(x)),\ x\in\Omega\subset\mathbb{R}^{4},\label{eq:PDE}
\end{equation}
and 
\begin{equation}
\sigma_{2}(A(\rho,u(x)))\ge f(u(x)),\ x\in\Omega\subset\mathbb{R}^{4},\label{eq:inequality}
\end{equation}
where $f(u)>0.$

Let
\begin{align*}
\Gamma_{2}^{+} & =\{A\in{\rm Sym^{2}(\mathbb{R}^{4})},\ \ \sigma_{1}(A)>0,\,\sigma_{2}(A)>0\},\\
\Gamma_{2}^{-} & =\{A\in{\rm Sym^{2}(\mathbb{R}^{4})},\ \ \sigma_{1}(A)<0,\,\sigma_{2}(A)>0\}.
\end{align*}
We define positive and negative cones as
\begin{align*}
\mathcal{C}_{2}^{+} & :=\{u\in C^{2}(\Omega),\ A(\rho,u(x))\in\Gamma_{2}^{+},\ x\in\Omega\},\\
\mathcal{C}_{2}^{-} & :=\{u\in C^{2}(\Omega),\ A(\rho,u(x))\in\Gamma_{2}^{-},\ x\in\Omega\}.
\end{align*}
It is known that when $u$ satisfies (\ref{eq:PDE}) or (\ref{eq:inequality})
in a connected domain, then $u$ belongs to either $\mathcal{C}_{2}^{+}$or
$\mathcal{C}_{2}^{-}.$ However, solutions in different cones can
behave very differently. 

In \cite{Fang2018sigma_Y,Fang2020AP}, the first-named author and
the third-named author have studied the $\sigma_{2}$-curvature in
both the positive cone and the negative cone case, where special functions
are chosen for the right hand side of (\ref{eq:PDE}) and (\ref{eq:inequality}).
A major technique in \cite{Fang2018sigma_Y,Fang2020AP} is to construct
monotonic ``quasi-local mass'' type quantities that lead to interesting
geometric conclusions. In this paper, we would like to further explore
this construction, state the most general results, and apply them
to study solutions or super-solutions of related partial differential
equations. In particular, we have established the following:
\begin{thm}
\label{thm:mass-theorem}Consider a bounded domain $\Omega\subset\mathbb{R}^{4}$
with almost $C^{2}$-boundary. Assume $u\in\mathcal{C}_{2}^{+}\cap C^{2}(\overline{\Omega})$
satisfying differential inequality (\ref{eq:inequality}) and $u|_{\partial\Omega}=\tau$.
Define the level set $L_{t}=\{x:u(x)=t\}$ and $\Omega_{t}=\{x:u(x)>t\}$.
Let $F(t)$ be an anti-derivative of $f(t)$, and $|S^{3}|$ be the
area of a unit 3-sphere. Then, the quasi-local mass
\begin{align}
M(t) & :=-2\left(\frac{1}{|S^{3}|}\int_{\Omega_{t}}\sigma_{2}(A(\rho,u(x))dx\right)\left(\frac{1}{|S^{3}|}\int_{\Omega_{t}}{\rm div}(|\nabla u|^{2}\nabla u)dx\right)^{\frac{1}{3}}\label{eq:mass for supersolution in mass-theorem}\\
 & \ +\frac{\rho}{8}\left(\frac{1}{|S^{3}|}\int_{\Omega_{t}}{\rm div}(|\nabla u|^{2}\nabla u)dx\right)^{\frac{4}{3}}-\frac{12}{|S^{3}|}\int_{\Omega_{t}}\left(F(t)-F(u(x))\right)dx\nonumber 
\end{align}
is well defined for $t\in(\tau,\max_{\Omega}u)$ and is monotonically
increasing.

Suppose that $u\in\mathcal{C}_{2}^{+}\cap C^{3}(\Omega)\cap C^{2}(\bar{\Omega})$
and further satisfies the equality (\ref{eq:PDE}). Then, for almost
all $t\in(\tau,\max_{\Omega}u)$,
\begin{align}
M(t)= & \frac{9\rho}{8}\left((\frac{1}{|S^{3}|}\int_{L_{t}}|\nabla u|^{3}dl)^{\frac{4}{3}}+\frac{1}{|S^{3}|}\int_{L_{t}}|\nabla u|^{3}\langle x,\nabla u\rangle dl\right)\label{eq:long mass expression}\\
 & +\frac{1}{|S^{3}|}\int_{L_{t}}H(x)|\nabla u|^{2}\left((\frac{-1}{|S^{3}|}\int_{L_{t}}|\nabla u|^{3}dl)^{\frac{1}{3}}-\langle x,\nabla u\rangle\right)dl,\nonumber 
\end{align}
where $H$ is the mean curvature of $L_{t}$.
\end{thm}

$M(t)$ in (\ref{eq:mass for supersolution in mass-theorem}) works
for super-solutions satisfying (\ref{eq:inequality}) and the monotonicity
is independent of $\rho$. It covers different types of mass quantities
that appeared in \cite{Fang2018sigma_Y,Fang2020AP} and can be applied
to 2-Hessian equations. For solutions to (\ref{eq:PDE system}), the
expression of $M(t)$ in (\ref{eq:long mass expression}) solely depends
on the geometry of the level set hypersurfaces and the gradient of
$u$. It is a major improvement compared to \cite{Fang2018sigma_Y,Fang2020AP}.
(\ref{eq:long mass expression}) is achieved by establishing a Pohozaev
type identity which can be stated in a more general form. See the
appendix for details.

An immediate application of (\ref{eq:long mass expression}) is the
following:
\begin{thm}
\label{thm:Dirichlet}Let $\Omega$ be a round ball in $\mathbb{R}^{4}$.
Assume that for $\rho\leq0,$ $u\in\mathcal{C}_{2}^{+}$ satisfies
(\ref{eq:PDE}), $f(u)$ is integrable, and $u|_{\partial\Omega}=\tau$.
Then $u$ is rotationally symmetric in $\Omega.$
\end{thm}

We point out that our method extends that of Lions in \cite{lionst1981two}
for the Laplacian in dimension 2 and we have not assumed the regularity
of function $f$ of (\ref{eq:PDE}). When $f$ has higher regularity,
regardless of the sign of $\rho,$ it is possible to establish the
result of Theorem \ref{thm:Dirichlet} using the moving plane method.
The approach here does not rely on the regularity of $f$; however,
it works only for $\rho$ being non-positive.

\textcolor{black}{As a special case of (\ref{eq:PDE}), consider the
following PDE:}
\begin{align}
\begin{cases}
\sigma_{2}(A(\rho,u))=\exp(\beta u)p(u), & x\in\mathbb{R}^{4}\\
\sigma_{1}(A(\rho,u))>0,
\end{cases}\label{eq:PDE system}
\end{align}
where $\beta>0$ and $p$ is a smooth positive polynomial-like function
satisfying (\ref{eq:growth condition of P(u)}) and (\ref{eq:growth condition 2}).
In particular, $p$ may be chosen as a positive polynomial or simply
a positive number. Noting that when $\rho=1,$ $p=\frac{3}{2}$ and
$\beta=4$, (\ref{eq:PDE system}) is the so-called $\sigma_{2}$-Yamabe
problem for the round sphere, which is well studied in conformal geometry.
It is also worth pointing out that there is a natural transformation
of a problem in the negative cone into a problem in the positive cone.
See Section 2 for more details. Especially, when $\rho=-1,$ $p=\frac{3}{2}$
and $\beta=-4,$ (\ref{eq:PDE system}) is equivalent to the $\sigma_{2}$-Yamabe
problem for the hyperbolic 4-space, which has been studied in \cite{Fang2020AP}.
When $\rho=0,$ (\ref{eq:PDE system}) represents the 2-Hessian equation
of critical degree. 

We use the mass monotonicity formula (\ref{eq:long mass expression})
together with careful analysis of the geometry of level sets to establish
a Liouville type theorem for equation (\ref{eq:PDE system}). We now
state our main theorem as follows:
\begin{thm}
\label{thm:Maintheorem} For $\rho>0$, if $u\in\mathcal{C}_{2}^{+}$
is a $C^{2}$-solution of (\ref{eq:PDE system}), then $u$ is rotationally
symmetric with respect to a point.
\end{thm}

\begin{rem}
\textcolor{black}{By rescaling, we may assume $\beta=4$. See Section
2, Remark \ref{rem:+/-cone}. Also, if $\rho\geq2$, there is no entire
solution to (\ref{eq:PDE system}) on the whole space $\mathbb{R}^{4}$.
See Corollary \ref{cor:If rho>=00003D2}. }When $p$ is a positive
constant and $\rho\in(0,2),$ rotationally symmetric solutions of
(\ref{eq:PDE system}) exist. See Section \ref{sec:Radial-solutions}
for details. 
\end{rem}

Our study and results are influenced by the studies of $\sigma_{k}$-curvature
in conformal geometry and $k$-Hessian equations in the theory of
non-linear differential equations, which have been developed with
intertwining themes in the past few decades. 

In the context of conformal geometry, $\sigma_{k}$-curvature was
first studied by Viaclovsky \cite{V1} as a natural extension of the
scalar curvature. Let $(M^{n},g)$ be a Riemannian manifold and 
\[
A_{g}:=\frac{1}{n-2}\left(\operatorname{Ric}_{g}-\frac{R_{g}}{2(n-1)}\cdot g\right)
\]
be its Schouten tensor, where ${\rm Ric}_{g}$ is the Ricci curvature,
and $R_{g}$ is the scalar curvature. $\sigma_{k}$-curvature, denoted
as $\sigma_{k}(g^{-1}A_{g})$, is the degree $k$ elementary symmetric
polynomial of the eigenvalues of $g^{-1}A_{g}$. $\sigma_{k}$-Yamabe
problems are analog to the classic Yamabe problem in which we seek
a conformal metric $g_{u}=\exp(2u)g$ such that
\begin{align}
\sigma_{k}(g_{u}^{-1}A_{g_{u}}) & =const,\label{eq:k-Yamabe problem}\\
\sigma_{1}(g_{u}^{-1}A_{g_{u}}), & \cdots,\sigma_{k}(g_{u}^{-1}A_{g_{u}})>0.\nonumber 
\end{align}
Equation (\ref{eq:k-Yamabe problem}) has significant geometric applications.
See \cite{V1,V4,V2,V3,CGY2chang2002priori,CGY3chang2003entire,GW2,G0gonzalez2004singular,CHYChang2005ClassificationOS,G1Gonz05,wang2006priori,G2Gonz06cv,STWsheng2007yamabe,Li2014ACT,Fang2018sigma_Y}
for some incomplete references in this field. In particular, the work
of Chang-Gursky-Yang \cite{CGY1chang2002equation,CGY4chang2003conformally}
explores properties of $\sigma_{2}$-curvature in closed 4-manifolds
and gives a conformal characterization of 4-spheres. The majority
of geometric applications in this direction require a positive cone
condition. In \cite{Fang2020AP}, a special negative cone case is
studied. In particular, the standard hyperbolic 4-space can be characterized
as the sharp case of the quasi-local mass inequality.

When $n\geq3,k\geq1$, Liouville type theorems on $\mathbb{R}^{n}$
for the $\sigma_{k}$-Yamabe equations have been established. A powerful
analytic device is the moving plane/sphere method. Let $g_{u}=u^{\frac{4}{n-2}}|dx|^{2}$
be a solution to (\ref{eq:k-Yamabe problem}). For $k=1$, it is proved
in the celebrated paper \cite{CGS89} by Caffarelli-Gidas-Spruck that
$u$ is rotationally symmetric. See also Li-Zhang \cite{Li-Zhang03}
for an alternative proof. Under the assumption that $u(x)=O(|x|^{2-n})$,
Viaclovsky in \cite{V4} proved the Liouville theorem for $k\geq2$.
Li-Li in \cite{Li-Li03,Li-Li05} demonstrated the Liouville type theorems
for a class of conformally invariant equations including $\sigma_{k}$-Yamabe
equation (\ref{eq:k-Yamabe problem}). 

\textcolor{black}{If $k=2$ and $n\geq4$, Chang-Gursky-Yang \cite{CGY2chang2002priori,CGY3chang2003entire}
classified the solution to (\ref{eq:k-Yamabe problem}) without using
the moving plane/sphere method. The classification in \cite{CGY2chang2002priori,CGY3chang2003entire}
requires an additional integrability condition for $n\geq6$. However,
the integrability condition is not required for $n=4,5$. }

For $n=2,k=1$, instead of (\ref{eq:k-Yamabe problem}), one considers
the problem of prescribing constant Gaussian curvature with the equation:
\begin{equation}
-\Delta u=e^{u}\text{ on }\mathbb{R}^{2}.\label{eq:prescribing Gaussian}
\end{equation}
Chen-Li in \cite{chen1991classification} proved that under the integrability
assumption$\int_{\mathbb{R}^{2}}e^{u}dx<\infty,$ the only $C^{2}$
solution to (\ref{eq:prescribing Gaussian}) are rotationally symmetric.
In a recent paper \cite{li2020liouville}, Li-Lu-Lu extended Chen-Li's
results to M\"{o}bius invariant equations in $\mathbb{R}^{2}$ without
the integrability assumption similar to the case in \cite{CGY2chang2002priori,CGY3chang2003entire}.

$k$-Hessian equations are intermediate cases between Laplacian equations
and \MA equations, with many applications in geometric analysis and
other fields. Consider the positive solution to the equation 
\begin{align}
\sigma_{k}(-\nabla^{2}u) & =u^{p};\ \sigma_{1}(-\nabla^{2}u),\cdots,\sigma_{k}(-\nabla^{2}u)>0.\label{eq:k-hessian exponents}
\end{align}
Tso in \cite{tso1990remarks} indicated that $k^{*}=\frac{(n+2)k}{n-2k}$
for $k<\frac{n}{2}$ is the critical exponent and proved a non-existence
theorem for supercritical exponent $p>k^{*}$ on star-shaped domains.
The critical exponents are related to the Sobolev inequality of $k$-Hessian
equations, which was studied by Wang in \cite{wang94}.  In the dimensional
critical case $n=2k$, Tian-Wang in \cite{tian-wang05} proved a Moser-Trudinger
type inequality. Notably, sharp constants of the Sobolev and Moser-Trudinger
type inequalities in \cite{wang94,tian-wang05} are attained for entire
functions on $\mathbb{R}^{n}$. See the lecture notes of Wang \cite{wang2009k}.
Compared to classical Sobolev inequalities, Liouville type theorems
for $k$-Hessian equations are expected.  When $k=1$, it is a classical
result of Gidas-Spruck \cite{Gidas-spruck81} that (\ref{eq:k-hessian exponents})
only admits zero solutions when $p<k^{*}$. Using the nonlinear potential
theory, Phuc-Verbitsky \cite{phuc2008quasilinear} extended Gidas-Spruck's
result for $p\le\frac{nk}{n-2k}<k^{*}$. See also an alternative proof
by Ou \cite{Ou2010} using methods developed by Gonz$\text{{á}}$lez
\cite{G1Gonz05,G2Gonz06cv}. For $\frac{nk}{n-2k}<p\le k^{*}$, the
classification of positive solutions to $\sigma_{k}(-\nabla^{2}u)=u{}^{p}$
in $\mathbb{R}^{n}$ remains open. The classification of radial solutions
in the entire space can be found in \cite{Miyamoto}\cite{wang2019critical}
and references therein.

For constant $k$-Hessian equations, Liouville type theorems (also
called Bernstein type theorems in the literature) have been established
in various settings. Consider
\begin{equation}
\sigma_{k}(-\nabla^{2}u)=1;\ \sigma_{1}(-\nabla^{2}u),\cdots,\sigma_{k}(-\nabla^{2}u)>0.\label{eq:k-hessian const}
\end{equation}
One may expect that (\ref{eq:k-hessian const}) only admits solutions
of quadratic functions and it has been confirmed in many cases: the
classic J$\ddot{o}$rgens-Calabi-Pogorelov theorem \cite{jorgens1954losungen,calabi1958improper,pogorelov1972improper}
for \MA equation; \cite{BCGJ2003} on convex solutions of (\ref{eq:k-hessian const});
$n=3,k=2$ with a quadratic growth condition by Warren-Yuan \cite{warren2009hessian};
$k=2$ by Chang-Yuan \cite{changyuan10} and Shankar-Yuan \cite{shankar2021rigidity}
with semi-convexity assumptions; Li-Ren-Wang \cite{Li-ren-wangjfa16}
with $k+1$-convexity and quadratic growth conditions; $k=2$ by Chen-Xiang
in \cite{chen-xiangjde19} assuming a quadratic growth condition and
a lower bound of $\sigma_{3}\left(-\nabla^{2}u\right)$. The Liouville
type theorem to (\ref{eq:k-hessian const}) in general does not hold
without additional convexity or growth conditions, since there are
some non-polynomial-like solutions to (\ref{eq:k-hessian const})
demonstrated by Warren \cite{warren2016nonpolynomial} and Li \cite{li2021non}.

Though leading terms of (\ref{eq:k-Yamabe problem}), (\ref{eq:k-hessian exponents}),
and (\ref{eq:k-hessian const}) look similar, each of these equations
has unique features. In the case of $\sigma_{k}$-Yamabe equation
(\ref{eq:k-Yamabe problem}) on $\mathbb{R}^{n}$, the eigenvalues
of the Schouten tensor are invariant under the M\"{o}bius group.
A special element in the M\"{o}bius group is the inversion map, which
reduces the study of asymptotic behavior of solutions at infinity
to that of solutions with an isolated point singularity. Once an estimate
near the point singularity is established, by the inversion map, we
will obtain the information of solutions near infinity to initiate
the moving plane/sphere method. See \cite{CGS89,Korevaar1999RefinedAF,Li-Li03,Li-Li05,HLTHan2010AsymptoticBO,han2021asymptotic}
for details. On the other hand, the M\"{o}bius group, in general,
does not preserve the $k$-Hessian equations, which causes some difficulties
to understand the asymptotic behavior of corresponding solutions.
When $\rho\not=1$ in equation (\ref{eq:PDE system}), the M\"{o}bius
invariance is also not available as the $k$-Hessian equations. It
is worth noting that the affine invariance of \MA equations simplifies
the proof of the corresponding Liouville theorems, c.f. \cite{Caffarelli2003AnET}.

Another feature differing $\sigma_{k}$-Yamabe equations from $k$-Hessian
equations is the formula of local $C^{1},C^{2}$ estimates. Instead
of dealing with the eigenvalues of $\nabla^{2}u$, with the help of
some lower order terms, one magically obtains certain a priori $C^{1}$
and $C^{2}$ estimates independent of the lower bound of $u$. Such
estimates have been established by Chen \cite{Chen2005LocalEF} for
a class of fully nonlinear equations including $\sigma_{k}$-Yamabe
equations and our equation (\ref{eq:PDE system}).  See Section \ref{sec:local-estimates}
for details. See also Guan-Wang \cite{guanWan2003}, Wang \cite{wang2006priori},
Li \cite{li2009local} for gradient estimates of equations similar
to $\sigma_{k}$-Yamabe equations.

\textcolor{black}{We briefly explain our approach to prove Theorem
\ref{thm:Maintheorem}. Our method needs $3$ crucial estimates: a
total integral upper bound, a local $C^{1}$ and $C^{2}$ estimate,
and a decay estimate from an $\epsilon$-regularity type argument.}\textcolor{red}{{}
}\textcolor{black}{These estimates are well known for experts for
$\sigma_{2}$-Yamabe equations and one of our key observations is
that when $\rho$ is positive, they are applicable to equation (\ref{eq:PDE system}).
For the total integral bound, we follow methods from Chang-Gursky-Yang
\cite{CGY3chang2003entire}; for the local $C^{1}$ and $C^{2}$ estimate,
we use theorems of Chen \cite{Chen2005LocalEF} directly.}\textcolor{red}{{}
}\textcolor{black}{Both estimates essentially depend on the positivity
of $\rho$.}\textcolor{red}{{} }\textcolor{black}{Combining the first
$2$ estimates, we generalize an $\epsilon$-regularity estimate by
Guan-Wang \cite{guanWan2003} for some general function $f(u)$ in
the form of (\ref{eq:PDE system}), which leads to the decay estimate.
See Corollary \ref{cor:Guan-Wang} for details.}

\textcolor{black}{The main technical contribution of this paper is
to combine the above mentioned estimates with our mass formula in
Theorem \ref{thm:mass-theorem} to establish the uniqueness and the
symmetry of solutions to (\ref{eq:PDE system}). In fact, with the
above mentioned estimates, and} some techniques in Li \cite{li2009local}
and Li-Nugyen \cite{li2014harnack}, we can prove that \textcolor{black}{the
level set after proper rescaling converges to sphere in Gromov-Haus}dorff
sense. We then evaluate the corresponding mass quantity for a sequence
of enlarging level sets and shows that it has to be constant, which
implies the rotational symmetry. Overall, our mass inequality gives
a direct way to compare geometric and analytical information at infinity
and at local maximal points of $u$, which leads to our theorem.

Our approach has several advantages. First, only information of a
sequence of level sets near infinity is needed to obtain the desired
comparison result. See details in sections \ref{sec:Asymptotic-behavior}
and \ref{sec:Proof-of-main theorem}. Second, the mass quantity (\ref{eq:long mass expression})
only involves boundary integrals and is invariant under rescaling.
Note that we do not use the M\"{o}bius  transform, which is essential
in earlier works of Yamabe type equations in conformal geometry. Third,
our approach is center-point free since the symmetry of solutions
can be deduced from the rigidity of the isoperimetric inequality.
Last, the mass quantity works for Hessian equations, which will be
the focus of our future work.

Naturally, one of our next goals is to consider the dimension 4 case
of the following conjecture for $k$-Hessian equations, which is well
known among experts in the field:
\begin{conjecture}
\label{conj: 2hessian liouville}\textcolor{black}{For $\rho=0$ and
$n=2k$, the equation (\ref{eq:PDE system}), re-written as
\[
\begin{cases}
\sigma_{k}(-\nabla^{2}u)=e^{nu},\\
\sigma_{1}(-\nabla^{2}u),\cdots,\sigma_{k}(-\nabla^{2}u)>0,
\end{cases}
\]
 has only rotationally symmetric solutions on $\mathbb{R}^{n}$.}
\end{conjecture}

Note that Theorem \ref{thm:Maintheorem} indicates a possible approach
to solve Conjecture \ref{conj: 2hessian liouville} in dimension 4,
though several of our key estimates fail when $\rho=0.$ Apart from
Conjecture \ref{conj: 2hessian liouville}, when $\rho>0$, $n=2k>4$,
the Liouville theorem similar to Theorem \ref{thm:Maintheorem} is
unknown. It is interesting to see if our current approach may be extended
to general dimensions. If $n>2k$, Liouville type theorems for (\ref{eq:k-hessian exponents})
are also expected but yet to be established. The $k$-Hessian equation
in Conjecture \ref{conj: 2hessian liouville} is a limiting case for
either $\rho>0$ or $\rho<0$.

When $\rho<0$, $A(\rho,u)\in\Gamma_{2}^{+}$ is equivalent to the
negative cone condition: $A(-\rho,-u)\in\Gamma_{2}^{-}$. See Section
2 for more details. For the $\sigma_{2}$-Yamabe equation in dimension
4, mass monotonicity formulae have been established by Fang-Wei \cite{Fang2018sigma_Y,Fang2020AP}
in both positive and negative cone.  It is reasonable to find some
parallel theories for general negative $\rho$. We also expect that
the mass, defined in \cite{Fang2020AP}, of the asymptotically hyperbolic
metric for $\sigma_{2}$ in 4-manifolds holds in multiply connected
domain. This is indicated by the work of Shen-Wang \cite{SHEN2021109228},
which has explored a mass type quantity for $-\Delta u=e^{u}$ by
complex analysis in a multiply connected domain in dimension 2.

The rest of this paper is organized as follows. In Section \ref{sec:Preliminary-discussion},
we state some preliminary facts regarding the $A(\rho,u)$ tensor.
In Section \ref{sec:Mass-monotonicity-formulae}, we restate the mass
construction of \cite{Fang2018sigma_Y,Fang2020AP} and state a general
monotonicity formula. In Section \ref{sec:Dirichlet-problem-for},
we provide an application for the mass monotonicity formula and prove
the symmetry of solutions of a Dirichlet problem. \textcolor{black}{In
Section \ref{sec:Bounded-Integral}, we establish an integral bound
using Chang-Gursky-Yang \cite{CGY3chang2003entire}.} In Section \ref{sec:local-estimates},
we adopt the argument of Guan-Wang in \cite{guanWan2003} to describe
decay estimates of $u$ near infinity. In Section \ref{sec:Asymptotic-behavior},
using the estimates in section \ref{sec:local-estimates}, we establish
$C^{1,\alpha}$ convergence for a sequence of level sets of $u$.
In Section \ref{sec:Proof-of-main theorem}, we use the Pohozaev identity
and the monotonicity formula to prove the main theorem. In Section
\ref{sec:Radial-solutions}, we discuss the existence and uniqueness
of radial solutions for certain cases. In Appendix, we state and prove
a Pohozaev identity using 2 different methods.

\section{Preliminary discussion \label{sec:Preliminary-discussion}}

In this section, we list some elementary facts regarding our symmetric
tensor $A(\rho,u)$ and the corresponding $\sigma_{k}$ quantity.
From now on, we assume that $n=4.$
\begin{rem}
\label{rem:+/-cone}Direct computation shows the following
\[
A(-\rho,-u)=-A(\rho,u).
\]
Thus, all results regarding the positive cone case can be stated for
the negative cone case, with a sign change of $\rho$ and vice versa.
For simplicity, we only discuss the positive cone case in this paper.
We state some scaling properties of $A$ here for future use:

\[
A(\rho,u(ax)+b)=a^{2}A(\rho,u),
\]
\[
A(\rho,u)=\frac{1}{\rho}A(1,\rho u),
\]
for $\rho\neq0$.
\end{rem}

The divergence structure of $A(\rho,u)$ is of crucial importance
in our proof. It is easy to check that 
\begin{align}
\sigma_{2}(A(\rho,u)) & =\sigma_{2}(\nabla^{2}u)+\frac{\rho}{2}\partial_{i}(|\nabla u|^{2}u_{i})\nonumber \\
 & =-\frac{1}{2}\partial_{i}[(-\Delta u\delta_{ij}+u_{ij}-\rho|\nabla u|^{2}\delta_{ij})u_{j}].\label{eq:divergence structure}
\end{align}

In \cite{li2014harnack}, Li-Nguyen have studied the $\sigma_{k}$-Yamabe
problem and shown that $A(1,\underline{u}(r))\in\overline{\Gamma_{2}^{+}}$
in viscosity sense. Following \cite{li2009local,li2014harnack}, we
define the corresponding viscosity solution for (\ref{eq:PDE system}):
\begin{defn}
[Viscosity solutions] \label{def:Viscosity definition}Let $w$ be
a lower semi-continuous (respectively, upper-semi-continuous ) function
in $\Omega$ . We say that 
\[
A(\rho,w)\in\overline{\Gamma_{2}^{+}},\quad({\rm resp.}\ A(\rho,w)\in\mathrm{Sym}^{2}(\mathbb{R}^{4})\backslash\Gamma_{2}^{+}),
\]
in the \emph{viscosity sense} if for all $x_{0}\in\Omega$, $\varphi\in C^{2}(\Omega)$,
$w-\varphi(x_{0})=0$ and $w-\varphi\ge0$ (resp. $w-\varphi\leq0$)
near $x_{0}$, it holds 
\[
A(\rho,\varphi))|_{x_{0}}\in\overline{\Gamma_{2}^{+}},\quad({\rm resp.}\ A(\rho,\varphi)\in\mathrm{Sym}^{2}(\mathbb{R}^{4})\backslash\Gamma_{2}^{+}).
\]
For a continuous function $w$ in $\Omega$, we say $A(\rho,w)\in\partial\Gamma_{2}^{+}$
in the\emph{ }viscosity sense\emph{ }if $A(\rho,w)\in\overline{\Gamma_{2}^{+}}\backslash\Gamma_{2}^{+}$.
If $A(\rho,w)\in\overline{\Gamma_{2}^{+}}$, we call $w$ a \emph{viscosity
super-solution} to the equation $A(\rho,w)\in\partial\Gamma_{2}^{+}$. 
\end{defn}

It is straightforward to check that a $C^{2}$ function satisfying
relations mentioned above in the viscosity sense satisfies the corresponding
relations in the classic sense. 

Let 
\[
\underline{u}(r):=\min_{|x|=r}u(x).
\]
Since super-solutions are preserved when taking minimum (see \cite{Caffarelli1995FullyNE},
chapter 2), a straightforward argument shows the following:
\begin{lem}
\label{lem-aviscosity radial solution}$A(\rho,\underline{u}(|x|))\in\overline{\Gamma_{2}^{+}}$
in the viscosity sense. 
\end{lem}

We now state an important monotonicity formula for $\underline{u}(r)$
following a construction of Li-Nguyen \cite{li2014harnack}:
\begin{thm}
[\cite{li2014harnack}, Lemma 2.10]\label{lem:Li-Nguyen Monotonicity formula }For
$r>r_{0}>0$, 
\begin{equation}
r\mapsto\frac{\underline{u}(r)-\underline{u}(r_{0})}{\ln r-\ln r_{0}}\label{eq:lemma monontonicity formula}
\end{equation}
is non-increasing. In other words, $\underline{u}\circ\exp:(-\infty,\infty)\to\mathbb{R}$
is concave.
\end{thm}

The proof of Theorem \ref{lem:Li-Nguyen Monotonicity formula } is
almost the same as that appeared in [\cite{li2014harnack}, Lemma 2.10].
We omit it here.

\section{Mass monotonicity formulae\label{sec:Mass-monotonicity-formulae}}

In this section, we generalize some results from \cite{Fang2018sigma_Y,Fang2020AP}
regarding the construction of a monotonic quantity along level sets.
We assume that $u\in\mathcal{C}_{2}^{+}$ satisfies (\ref{eq:inequality}).
Note that some of the notations that we use here are different from
those in \cite{Fang2018sigma_Y,Fang2020AP}.
\begin{rem}
\label{rem:signconvention}If $u\in\mathcal{C}_{2}^{+},$ then $\sigma_{1}(A(\rho,u))=-\Delta u-\rho|\nabla u|^{2}>0,$
which, by the maximum principle, indicates that $u$ has no interior
minimums; similarly, if $u\in\mathcal{C}_{2}^{-},$ then $u$ has
no interior maximums.
\end{rem}

Let $\Omega$ be either $\mathbb{R}^{4}$ or a bounded domain in
$\mathbb{R}^{4}$, and assume that $u$ satisfies (\ref{eq:inequality})
in $\Omega.$ When $\Omega$ is bounded, we further assume that $u|_{\partial\Omega}=\tau.$
For $t\in\mathbb{R},$ define the following sets

\[
L_{t}:=\{x\in\Omega:u(x)=t\},\ \ \Omega_{t}:=\{x\in\Omega:u(x)>t\}.
\]
We note that when $\Omega=\mathbb{R}^{4}$, if $f(u)$ is integrable
in $\mathbb{R}^{4},$ a similar argument of Proposition 3.6 in Guan-Wang
\cite{guanWan2003} in our settings shows that $\lim_{x\to\infty}u(x)=-\infty$.
See Corollary \ref{cor:Guan-Wang} for details. Thus by maximum principle,
without loss of generality, we make the following assumptions:
\begin{enumerate}
\item $\tau':=\max_{\Omega}u>\tau$ is attained inside $\Omega$. Define
\[
\mathcal{I}=\begin{cases}
(-\infty,\tau'), & {\rm if}\ \Omega=\mathbb{R}^{4},\\
(\tau,\tau'), & {\rm if}\ \Omega\ {\rm is\ bounded}.
\end{cases}
\]
\item $\Omega_{t}$ and $L_{t}$ are non-empty for $t\in\mathcal{I}$. 
\item $\Omega_{t}$ and $L_{t}$ are bounded sets in $\Omega$. 
\end{enumerate}
From Lemma 10 in \cite{Fang2018sigma_Y}, we have 
\begin{lem}
Let
\[
\mathcal{S}:=\{x\in\mathbb{R}^{4}:\nabla u(x)=0\}.
\]
If $\sigma_{2}(A(\rho,u))\not=0$ in $\Omega$, $\mathcal{S}\cap\Omega$
has Hausdorff dimension at most $2$. \label{lem:Structure of S} 
\end{lem}

Locally, $L_{t}\backslash\mathcal{S}$ is a regular $C^{2}$ hypersurface.
As $u$ is continuous, $\partial\Omega_{t}\subset L_{t}$ and $L_{t}=\partial\Omega_{t}\cup\left(\mathcal{S}\cap L_{t}\right)$.
By Lemma \ref{lem:Structure of S}, $L_{t}\cap\mathcal{S}$ has co-dimension
at least $1$ in $L_{t}$ and 
\begin{equation}
0=\mathcal{H}^{3}(L_{t}\cap\mathcal{S})=\mathcal{H}^{3}\left(L_{t}\backslash\partial\Omega_{t}\right),\label{eq:zero measure}
\end{equation}
where $\mathcal{H}^{3}$ is the 3-dimensional Hausdorff measure. Therefore,
the singular part of $\partial\Omega_{t}$ has zero $\mathcal{H}^{3}$-measure,
implying $\Omega_{t}$ has almost $C^{2}$ boundary. 

We use $dl$ to denote the $\mathcal{H}^{3}$-measure on $L_{t}$
and it agrees with the volume element on $L_{t}\backslash\mathcal{S}$.
By (\ref{eq:zero measure}), we may use $\partial\Omega_{t},L_{t},L_{t}\backslash\mathcal{S}$
interchangeably when taking integrals of measurable functions with
respect to $dl$. 

In $L_{t}\backslash\mathcal{S}$, let $\nu$ be the outward normal
vector of $L_{t},$ then by Remark \ref{rem:signconvention}, 
\begin{equation}
\nu=-\frac{\nabla u}{|\nabla u|}.\label{eq:nu def}
\end{equation}
Fix a regular point $x\in L_{t}\backslash\mathcal{S}$. In a small
open neighborhood $U$ of $x$, we define, for any $x'\in U$
\[
y^{4}(x'):=-{\rm sgn}(u(x')-t){\rm dist}(x',L_{t}).
\]
We may choose a normal coordinate $y^{1},y^{2},$$y^{3}$ on $L_{t}$
near $x$ and extend them in $U$ so that $\langle\frac{\partial}{\partial y^{i}},\frac{\partial}{\partial y^{j}}\rangle(x)=\delta_{ij}$
and $\langle\frac{\partial}{\partial y^{i}},\frac{\partial}{\partial y^{4}}\rangle|_{U}=\delta_{i4}.$
In particular, by the choice of $y^{4},$ we have $\frac{\partial}{\partial y^{4}}=\nu.$
Let $h_{\alpha\beta}$ be the second fundamental form of the level
set $L_{t}$ with respect to $-\nu$. Using the Gauss-Weingarten formula
in our coordinate system $\{y^{i}\}$, we obtain 
\begin{equation}
A(\rho,u)=\left(\begin{array}{cccc}
 &  &  & -\nabla_{41}u\\
 & h_{\alpha\beta}|\nabla u|-\frac{\rho}{2}|\nabla u|^{2}\delta_{\alpha\beta} &  & -\nabla_{42}u\\
 &  &  & -\nabla_{43}u\\
-\nabla_{41}u & -\nabla_{42}u & -\nabla_{43}u & -\nabla_{44}u+\frac{\rho}{2}|\nabla u|^{2}
\end{array}\right),\label{eq:expression of A}
\end{equation}
and
\begin{equation}
-\Delta u=H|\nabla u|-u_{44}.\label{eq:new100}
\end{equation}
Here we use $\nabla_{ij}u=u_{ij}$ to denote the covariant derivatives
of $u$ with respect to our chosen coordinate system. For future use,
define a symmetric $3\times3$ matrix as 
\[
\widetilde{A}(x):=(h_{\alpha\beta}|\nabla u|-\frac{\rho}{2}|\nabla u|^{2}\delta_{\alpha\beta}).
\]

With (\ref{eq:expression of A}), we prove a useful lemma.
\begin{lem}
\label{lem:divergence }If $\rho\geq0$ and $\sigma_{2}(A(\rho,u))>0$,
then in $\mathbb{R}^{4}\backslash\mathcal{S}$,
\begin{equation}
{\rm div}(|\nabla u|^{2}\nabla u)=3|\nabla u|^{2}\left(u_{44}-\frac{1}{3}H|\nabla u|\right)<0,\label{eq:useful lemma eq 1}
\end{equation}
\begin{equation}
\sigma_{2}(A(\rho,u))\leq\sigma_{1}(\tilde{A})(\frac{1}{3}H|\nabla u|-u_{44}).\label{eq:useful lemma eq 2}
\end{equation}
\end{lem}

\begin{proof}
Since $u\in\mathcal{C}_{2}^{+},$
\begin{equation}
\sigma_{1}(\tilde{A})=H|\nabla u|-\frac{3\rho}{2}|\nabla u|^{2}>0.\label{eq:sigma_1(tilde A) mean curvature is positive}
\end{equation}
From (\ref{eq:expression of A}) and Newton-MacLaurin inequality,
\begin{align*}
\sigma_{2}(A(\rho,u)) & =\sigma_{1}(\tilde{A})(-u_{44}+\frac{\rho}{2}|\nabla u|^{2})+\sigma_{2}(\tilde{A})-\sum_{1\leq i<4}u_{i4}^{2}\\
 & \leq\sigma_{1}(\tilde{A})\left(-u_{44}+\frac{\rho}{2}|\nabla u|^{2}\right)+\sigma_{2}(\tilde{A})\\
 & \leq\sigma_{1}(\tilde{A})\left(-u_{44}+\frac{\rho}{2}|\nabla u|^{2}\right)+\frac{1}{3}\sigma_{1}^{2}(\tilde{A})\\
 & =\sigma_{1}(\tilde{A})\left(-u_{44}+\frac{1}{3}H|\nabla u|\right).
\end{align*}
Since $\sigma_{2}(A(\rho,u))>0$ and $\sigma_{1}(\tilde{A})>0$, by
(\ref{eq:useful lemma eq 2}), we have $-u_{44}+\frac{1}{3}H|\nabla u|>0$.
Then,
\[
{\rm div}(|\nabla u|^{2}\nabla u)=3|\nabla u|^{2}\left(u_{44}-\frac{1}{3}H|\nabla u|\right)<0.
\]
\end{proof}
We choose $F(t)$ such that $F'(t)=f(t)$ and define the following
quantities:
\begin{align}
N(t) & :=\frac{1}{|S^{3}|}\int_{\Omega_{t}}\sigma_{2}(A(\rho,u(x))dx,\label{eq:NPQV definition}\\
P(t) & :=\frac{1}{|S^{3}|}\left(F(t)|\Omega_{t}|-\int_{\Omega_{t}}F(u(x))dx\right)=\frac{1}{|S^{3}|}\int_{\Omega_{t}}[F(t)-F(u(x))]dx,\nonumber \\
Q(t) & :=-\left(\frac{1}{|S^{3}|}\int_{\Omega_{t}}{\rm div}(-|\nabla u|^{2}\nabla u)dx\right)^{\frac{1}{3}},\nonumber \\
V(t) & :=\frac{1}{|S^{3}|}\int_{\Omega_{t}}dx=\frac{1}{|S^{3}|}|\Omega_{t}|.\nonumber 
\end{align}
We show that these quantities are absolutely continuous with respect
to $t$. In the following, $a.e.$ means almost everywhere with respect
to the Lebesgue measure on $\mathcal{I}$ unless otherwise mentioned. 
\begin{lem}
\label{lem:absolutely continuous }Let $\phi$ be a function such
that $\|\phi\|_{L^{\infty}(\Omega_{t})}<C$ for $t\in[t_{0},t_{1}]\subset\mathcal{I}$.
Then,
\[
m_{\phi}(t):=\int_{\Omega_{t}}\phi dx
\]
is an absolutely continuous function defined on $[t_{0},t_{1}]$.
Moreover, 
\[
m_{\phi}'(t)=-\int_{L_{t}}\frac{\phi}{|\nabla u|}d\mathcal{H}^{3}\quad a.e.
\]
\end{lem}

\begin{proof}
Since $\|\phi\|_{L^{\infty}(\Omega_{t})}<\infty$, and $\Omega_{t}$
is bounded, $m_{\phi}(t)$ is well defined. We first assume $\phi\geq0$.
By Lemma \ref{lem:Structure of S} and the co-area formula (Proposition
2.1 in \cite{Brothers1988}),
\begin{align}
m_{\phi}(t_{0})-m_{\phi}(t) & =\int_{\mathcal{\mathcal{S}}\cap u^{-1}((t_{0},t))}\phi d\mathcal{H}^{4}+\int_{t_{0}}^{t}\int_{L_{\tau}}\frac{\phi}{|\nabla u|}d\mathcal{H}^{3}d\tau\label{eq:coarea f2}\\
 & =\int_{t_{0}}^{t}\int_{L_{\tau}}\frac{\phi}{|\nabla u|}d\mathcal{H}^{3}d\tau.\nonumber 
\end{align}
Thus, $m_{\phi}$ is absolutely continuous. Since the left hand side
of (\ref{eq:coarea f2}) is finite, we have $\int_{L_{t}}\frac{\phi}{|\nabla u|}d\mathcal{H}^{3}<\infty\ \ a.e.$,
and 
\begin{equation}
m_{\phi}'(t)=-\int_{L_{t}}\frac{\phi}{|\nabla u|}d\mathcal{H}^{3}\quad a.e.\label{eq:a.e. differentiable}
\end{equation}

If $\phi$ has negative parts, we let $\phi^{+}=\max\{\phi,0\}$,
$\phi^{-}=\max\{-\phi,0\}$. Applying (\ref{eq:coarea f2}) to $\phi^{+}$
and $\phi^{-}$ shows that $m_{\phi^{+}}(t),m_{\phi^{-}}(t)$ are
both finite and absolutely continuous with derivatives $a.e.$ in
the form of (\ref{eq:a.e. differentiable}). We then conclude the
Lemma since $m_{\phi}=m_{\phi^{+}}-m_{\phi^{-}}$.  
\end{proof}
\begin{cor}
\label{cor:-N, P absolutely continuous and measure positive}$N,P,Q,V$
are absolutely continuous functions in $t$. Furthermore, $\mathcal{H}^{3}(L_{t})$
is positive and finite for almost all $t$ in $\mathcal{I}$. 
\end{cor}

\begin{proof}
$N,P,Q^{3},V$ are all integrals of continuous functions on $\Omega_{t}$.
So we can apply Lemma \ref{lem:absolutely continuous } directly.
By Lemma \ref{lem:divergence } and Lemma \ref{lem:Structure of S},
${\rm div}(|\nabla u|^{2}\nabla u)<0$ $\mathcal{H}^{4}$-$a.e.$
in $\Omega$. If $Q^{3}(t')=0$ for some $t'$, then $|\nabla u|$
vanishes identically on $\Omega_{t}$, contradicted to Lemma \ref{lem:Structure of S}.
Hence, $Q^{3}(t)<0$ in any closed interval in $\mathcal{I}$. Since
$Q^{3}(t)$ stays away from $0$, $Q(t)$ is also absolutely continuous. 

The second statement is a direct application of Lemma \ref{lem:absolutely continuous }
to $m_{|\nabla u|}(t)=\int_{\Omega_{t}}|\nabla u|dx$. Then, $m_{|\nabla u|}'(t)=\int_{L_{t}}1d\mathcal{H}^{3}<\infty$
$a.e$. Hence, $\mathcal{H}^{3}(L_{t})<\infty$ $a.e.$ By De Giorgi's
structure theorem (Theorem 15.9 in \cite{Maggi-book}) and Lemma 11,
the perimeter of $\Omega_{t}$, $P(\Omega_{t})=\mathcal{H}^{3}(L_{t})$
$a.e.$ Then, $\mathcal{H}^{3}(L_{t})>0$ $a.e.$, and the positive
of $\mathcal{H}^{3}(L_{t})$ is due to isoperimetric inequality (Theorem
14.1 in \cite{Maggi-book}). 
\end{proof}
As all functions defined in (\ref{eq:NPQV definition}) are absolutely
continuous, we can compute their derivatives.
\begin{lem}
\label{lem:prep}For almost all $t$ in $\mathcal{I}$, we have 
\begin{align}
Q(t) & =-\left(\frac{1}{|S^{3}|}\int_{L_{t}}|\nabla u|^{3}dl\right)^{\frac{1}{3}},\label{eq:Q(t)}\\
N(t) & =\frac{1}{2|S^{3}|}\int_{L_{t}}(H|\nabla u|^{2}-\rho|\nabla u|^{3})\ dl,\label{eq:new101}\\
N'(t) & =-\frac{1}{|S^{3}|}\int_{L_{t}}\frac{\sigma_{2}(A(\rho,u))}{|\nabla u|}dl,\label{eq:new102}\\
V'(t) & =-\frac{1}{|S^{3}|}\int_{L_{t}}\frac{1}{|\nabla u|}dl,\label{eq:new103}\\
P'(t) & =\frac{1}{|S^{3}|}f(t)V(t),\label{eq:new104}\\
Q'(t) & =\frac{1}{3Q^{2}}\cdot\frac{1}{|S^{3}|}\int_{L_{t}}(H|\nabla u|-3u_{44})|\nabla u|dl.\label{eq:new105}
\end{align}
Here $H$ is the mean curvature of regular points on $L_{t}$. 
\end{lem}

\begin{proof}
Since $\mathcal{H}^{3}(L_{t})<\infty$ for almost all $t$ in $\mathcal{I}$,
we can apply the generalized divergence theorem in $\Omega_{t}$ (see
\cite{evans1991measure}, Section 5.8, Theorem 1). Then, (\ref{eq:Q(t)})
follows from the generalized divergence theorem and (\ref{eq:nu def}).
(\ref{eq:new101}) is a direct application of the generalized divergence
theorem, (\ref{eq:divergence structure}), and (\ref{eq:new100}).
(\ref{eq:new102})-(\ref{eq:new105}) follow from the co-area formula
and direct computation. 
\end{proof}
We are now ready to define the monotonic quantity
\begin{equation}
M(t)=2N(t)Q(t)+\frac{\rho}{8}Q^{4}(t)-12P(t).\label{eq:MASS new def}
\end{equation}

\begin{rem}
The definition of $M(t)$ is a generalization of a similar construction
in \cite{Fang2018sigma_Y}. We choose a different scaling factor for
convenience.
\end{rem}

The following theorem has been proved for solutions to (\ref{eq:PDE system})
with $\rho=1$. For general $\rho$, we prove it for completeness
with minor changes.
\begin{thm}
\label{thm:mass theorem}Assume that on an open domain $\Omega$,
$u\in C^{2}(\Omega)$ satisfies 
\begin{equation}
\begin{cases}
\sigma_{2}(A(\rho,u))\geq f(u)>0,\\
A(\rho,u)\in\Gamma_{2}^{+}.
\end{cases}\label{eq:system for mass}
\end{equation}
If $\Omega\neq\mathbb{\mathbb{R}}^{4},$ we assume that $\Omega$
is bounded and $u|_{\partial\Omega}=\tau.$ Then for $t\in(\max_{\Omega}u,\tau)$,
$M(t)$ is non-decreasing and absolutely continuous with respect to
$t$.

Furthermore, if $M(t)$ is constant, then $u$ is a radial solution.
\end{thm}

\begin{proof}
$M(t)$ is absolutely continuous by Corollary \ref{cor:-N, P absolutely continuous and measure positive}.
From Lemma \ref{lem:divergence }, it holds on $L_{t}$ $\mathcal{H}^{3}$-$a.e.$
that 
\begin{align*}
f(u)\leq\sigma_{2}(A(\rho,u)) & \leq\sigma_{1}(\tilde{A})(\frac{1}{3}H|\nabla u|-u_{44}).
\end{align*}
Using the H\"older inequality and the isoperimetric inequality, 
\begin{align}
 & \left(\frac{1}{|S^{3}|}\int_{L_{t}}\frac{f(t)}{|\nabla u|}dl\right)^{2}\cdot\left(\frac{1}{|S^{3}|}\int_{L_{t}}\sigma_{1}(\tilde{A})|\nabla u|dl\right)\cdot\left(\frac{1}{|S^{3}|}\int_{L_{t}}(\frac{H}{3}|\nabla u|-u_{44})|\nabla u|dl\right)\nonumber \\
 & \ge f(t)^{3}|L_{t}|^{4}\frac{1}{|S^{3}|^{4}}\ \ge f(t)^{3}4^{3}|\Omega_{t}|^{3}\frac{1}{|S^{3}|^{3}}.\label{eq:107}
\end{align}

By Lemma \ref{lem:prep},
\[
\frac{1}{|S^{3}|}\int_{L_{t}}\sigma_{1}(\tilde{A})|\nabla u|dl=2N(t)+\frac{\rho}{2}Q^{3}(t).
\]
We use Lemma \ref{lem:prep} to rewrite (\ref{eq:107}) as
\[
4^{3}P'^{3}(t)\leq N'^{2}(t)(2N(t)+\frac{\rho}{2}Q^{3}(t))Q'(t)Q^{2}
\]
which, by using the inequality of arithmetic and geometric means,
leads to
\[
4P'\leq\frac{2}{3}(N'Q)+\frac{1}{3}(2N+\frac{\rho}{2}Q^{3})Q'=\frac{1}{3}(2NQ+\frac{\rho}{8}Q^{4})'.
\]
Checking with the definition of $M(t),$ we have proved that $M'(t)\geq0\ a.e.$
$t\in\mathcal{I}$.

If $M'(t')$ is well defined for some $t'$ and $M'(t')=0$, then
all inequalities are identities in the previous computations, implying
the level set $L_{t'}$ is a round 3-sphere. Thus, the solution $u$
is radial if $M(t)$ is constant for all $t$.
\end{proof}
\begin{rem}
By rescaling (section \ref{sec:Preliminary-discussion}, Remark \ref{rem:+/-cone})
and Theorem \ref{thm:mass theorem}, it is easy to recover different
types of mass inequalities that appeared in \cite{Fang2018sigma_Y,Fang2020AP}.
Rather than the solutions of (\ref{eq:PDE}), our construction in
(\ref{eq:mass for supersolution in mass-theorem}) works for super-solutions
satisfying (\ref{eq:inequality}). Furthermore, the monotonicity of
$M(t)$ is independent of $\rho$ and it can be applied to 2-Hessian
equations.
\end{rem}

Finally, we remark that similar to the $\sigma_{k}$-Yamabe problem
in conformal geometry, a Pohozaev type identity regarding our problem
can be established. See Theorem \ref{thm:pohozaev}. We include a
proof in the Appendix.

We apply Theorem \ref{thm:pohozaev} to revise $M(t),$ making it
dependent only on quantities on $L_{t}$:
\begin{cor}
\label{cor:Pohozaev for K=00003D1}Let $u$ be a solution to 
\begin{equation}
\begin{cases}
\sigma_{2}(A(\rho,u))=f(u)>0,\\
A(\rho,u)\in\Gamma_{2}^{+}.
\end{cases}\label{eq:system for mass-1}
\end{equation}
Assume that $\partial\Omega_{t}$ is smooth and $u\in C^{2}(\bar{\Omega}_{t})$;
or $\partial\Omega_{t}$ is almost $C^{2}$ and $u\in C^{3}(\Omega_{t})\cap C^{2}(\bar{\Omega}_{t})$.
Then for almost all $t$, we have
\[
\int_{\Omega_{t}}8[F(u(x))-F(t)]dx=\int_{\partial\Omega_{t}}\left(-\frac{3}{4}\rho|\nabla u|^{4}\langle x,\nu\rangle+\frac{2}{3}H|\nabla u|^{3}\langle x,\nu\rangle\right)dl.
\]
\begin{align}
M(t) & =\frac{9\rho}{8}\left(Q^{4}(t)+\frac{1}{|S^{3}|}\int_{L_{t}}|\nabla u|^{3}\langle x,\nabla u\rangle dl\right)\label{eq:masse expression}\\
 & \quad+\frac{1}{|S^{3}|}\int_{L_{t}}H(x)|\nabla u|^{2}\left(Q(t)-\langle x,\nabla u\rangle\right)dl.\nonumber 
\end{align}
Notice that all terms appeared in $M(t)$ can be expressed as some
integrals on $L_{t.}$ 
\end{cor}

\section{Dirichlet problem for balls \label{sec:Dirichlet-problem-for}}

In this section, we prove the symmetry of the solutions to the following
Dirichlet problem (\ref{eq:Dirichlet problem}) by the mass formula
and the Pohozaev identity.

By Remark \ref{rem:+/-cone}, for simplicity, we may assume that $\Omega=\{|x|<1\}\subset\mathbb{R}^{4}$
and $u$ is a $C^{2}(\overline{\Omega})$ solution to 
\begin{equation}
\begin{cases}
\sigma_{2}(A(\rho,u))=f(u)>0, & \Omega,\\
A(\rho,u)\in\Gamma_{2}^{+}, & \ \\
u=\tau, & \partial\Omega.
\end{cases}\label{eq:Dirichlet problem}
\end{equation}
By Theorem \ref{thm:pohozaev} and Theorem \ref{thm:mass theorem},
we prove the following theorem.
\begin{thm}
If $\rho\le0$, then every $C^{2}$ solution $u$ to equation (\ref{eq:Dirichlet problem})
with integrable $f(u)$ is radial. 
\end{thm}

\begin{proof}
Let $\tau'=\max_{\Omega}u.$ By the maximum principle, it is clear
that $\tau'\geq\tau$, and $M(\tau')=0.$ 

We compute $M(\tau)$. Note that on $\partial\Omega=S^{3}$, $H=3$
and $\nu=x$. By (\ref{eq:masse expression}) and (\ref{eq:MASS new def})
\begin{align*}
M(\tau) & =\frac{\rho}{8}Q^{4}+Q\frac{1}{|S^{3}|}\int_{L_{t}}(3|\nabla u|^{2}-\rho|\nabla u|^{3})-\frac{3}{2}\int_{S^{3}}\left(\frac{3}{4}\rho|\nabla u|^{4}-2|\nabla u|^{3}\right)dl\\
 & =\text{\ensuremath{\frac{9\rho}{8}Q^{4}+}}\frac{3}{|S^{3}|}\left(\int_{L(t)}|\nabla u|^{2}\right)Q-\frac{1}{|S^{3}|}\int_{\partial\Omega}\left(\frac{9}{8}\rho|\nabla u|^{4}-3|\nabla u|^{3}\right)dl\\
 & =\rho\frac{9}{8}\left[Q^{4}-\frac{1}{|S^{3}|}\int_{\partial\Omega}|\nabla u|^{4}\right]+3Q\left[\frac{1}{|S^{3}|}\int_{\partial\Omega}|\nabla u|^{2}dl-Q^{2}\right].
\end{align*}
It is easy to see that $Q(\tau)\leq0$,
\begin{align*}
Q^{4}(\tau)=\left(\frac{1}{|S^{3}|}\int_{\partial\Omega}|\nabla u|^{3}dl\right)^{\frac{4}{3}} & \leq\frac{1}{|S^{3}|}\int_{\partial\Omega}|\nabla u|^{4}dl,
\end{align*}
and 
\[
\frac{1}{|S^{3}|}\int_{\partial\Omega}|\nabla u|^{2}dl\leq\left(\frac{1}{|S^{3}|}\int_{\partial\Omega}|\nabla u|^{3}\right)^{\frac{2}{3}}.
\]
We conclude that $M(\tau)\ge0=M(\tau'$). By Theorem \ref{thm:mass theorem},
we obtain $M(t)\equiv0$ for all $t\in[\tau,\tau']$ and hence $u$
is rotationally symmetric. 
\end{proof}
\begin{rem}
By Remark \ref{rem:+/-cone}, same results hold for $A(\rho,u)\in\Gamma_{2}^{-}$
and $\rho\geq0.$

We point out here that on annular domains, under some regularity assumptions
of $f$, Wang-Bao in \cite{wang2015overdetermined} proved the radial
symmetric property of solutions to $k$-Hessian equation by the moving
plane method. 
\end{rem}

\section{\textcolor{black}{Bounded Integral of} $\sigma_{2}(A(\rho,u))$ \label{sec:Bounded-Integral}}

In this section, we establish an upper bound of $\int_{\mathbb{R}^{4}}\sigma_{2}(A(\rho,u))dx$,
assuming $A(\rho,u)$ in $\Gamma_{2}^{+}$. The argument is modified
from a similar one of Chang-Gurksy-Yang \cite{CGY3chang2003entire}.
We note that the positivity of $\rho$ here is essential. The main
result of this section is the following
\begin{thm}
\label{thm:Integral finit}If $A(\rho,u)\in\Gamma_{2}^{+}$, then
\[
\int_{\mathbb{R}^{4}}\sigma_{2}(A(\rho,u))dx<\frac{C}{\rho^{2}}<\infty,
\]
for some constant $C$.
\end{thm}

The proof uses the divergence structure of $\sigma_{2}(A(\rho,u))$
to perform integration by parts with a cut-off function. Let $\eta$
to be a smooth function such that
\[
\eta(x)=\begin{cases}
1 & x\in B_{R},\\
0 & x\in\mathbb{R}^{4}\backslash B_{2R},
\end{cases}
\]
and 
\begin{equation}
|\nabla\eta|^{2}+|\nabla^{2}\eta|<\frac{C_{0}}{R^{2}},\label{eq:cutoff eta gradient estimate}
\end{equation}
for some fixed constant $C_{0}.$ We first state two lemmas. 
\begin{lem}
It holds \label{lem: computation of int sigm_2}
\begin{align*}
2\int_{\mathbb{R}^{4}}\sigma_{2}(A(\rho,u))\eta^{4}dx= & \int_{\mathbb{R}^{4}}\left(\left(\nabla_{kj}^{2}\eta^{4}\right)u_{k}u_{j}-|\nabla u|^{2}\Delta\eta^{4}\right)dx\\
 & +\int_{\mathbb{R}^{4}}\rho\eta^{4}{\rm div}(|\nabla u|^{2}\nabla u)dx.
\end{align*}
\end{lem}

\begin{proof}
Recall the divergence structure of $\sigma_{2}(A(\rho,u))$,
\begin{equation}
\sigma_{2}(A(\rho,u))=-\frac{1}{2}\partial_{j}\left(\left(-\Delta u\delta_{ij}+u_{ij}\right)u_{i}-\rho|\nabla u|^{2}u_{j}\right).\label{eq:divstructure}
\end{equation}
We compute 
\begin{align}
-\int_{\mathbb{R}^{4}}\partial_{j}\left(\left(-\Delta u\delta_{ij}+u_{ij}\right)u_{i}\right)\eta^{4}dx & =\int_{\mathbb{R}^{4}}\left(-\Delta u\delta_{ij}+u_{ij}\right)u_{i}\nabla_{j}\eta^{4}dx.\label{eq:integration by parts}
\end{align}
Note 
\begin{align*}
\int_{\mathbb{R}^{4}}\nabla_{j}\eta^{4}u_{kk}u_{j}dx & =-\int_{\mathbb{R}^{4}}\left(\nabla_{kj}^{2}\eta^{4}u_{k}u_{j}+\nabla_{j}\eta^{4}u_{k}u_{jk}\right)dx\\
 & =-\int_{\mathbb{R}^{4}}\nabla_{kj}^{2}\eta^{4}u_{k}u_{j}dx+\int_{\mathbb{R}^{4}}\Delta\eta^{4}u_{k}u_{k}dx+\int_{\mathbb{R}^{4}}\nabla_{j}\eta^{4}u_{kj}u_{k}dx.
\end{align*}
Hence,
\begin{equation}
\int_{\mathbb{R}^{4}}\nabla_{j}\eta^{4}\left(u_{ij}u_{i}-u_{kk}u_{j}\right)dx=\int_{\mathbb{R}^{4}}\nabla_{kj}^{2}\eta^{4}u_{k}u_{j}dx-\int_{\mathbb{R}^{4}}\Delta\eta^{4}u_{k}u_{k}dx.\label{eq:inequality estimate}
\end{equation}
Thus, by (\ref{eq:divstructure}-\ref{eq:inequality estimate}), 
\begin{align*}
2\int_{\mathbb{R}^{4}}\eta^{4}\sigma_{2}(A(\rho,u)) & =\int_{\mathbb{R}^{4}}\left(\nabla_{kj}^{2}\eta^{4}u_{k}u_{j}-|\nabla u|^{2}\Delta\eta^{4}\right)dx\\
 & \quad+\int_{\mathbb{R}^{4}}\rho\eta^{4}{\rm div}(|\nabla u|^{2}\nabla u)dx.
\end{align*}
\end{proof}
\begin{lem}
Notations as above. We have \label{lem: integral gradient estimate  }
\[
\int_{\mathbb{R}^{4}}\left|\nabla_{kj}^{2}\eta^{4}u_{k}u_{j}-|\nabla u|^{2}\Delta\eta^{4}\right|dx\leq\frac{C_{1}}{\rho^{2}},
\]
where $C_{1}$ is a constant independent of $R$ and $u$.
\end{lem}

\begin{proof}
Notice that, since $A(\rho,u)\in\Gamma_{2}^{+}$,
\[
0<\sigma_{1}(A(\rho,u))=-\Delta u-\rho|\nabla u|^{2}.
\]
It implies 
\begin{align*}
\int_{\mathbb{R}^{4}}\rho|\nabla u|^{2}\eta^{2}dx & \leq-\int_{\mathbb{R}^{4}}\left(\Delta u\right)\eta^{2}dx\\
 & =\int_{\mathbb{R}^{4}}\nabla u\cdot\nabla\eta^{2}dx\\
 & \leq2\int_{\mathbb{R}^{4}}|\nabla u|\cdot|\nabla\eta|\eta dx\\
 & \leq\frac{\rho}{2}\int_{\mathbb{R}^{4}}|\nabla u|^{2}\eta^{2}dx+\frac{2}{\rho}\int_{\mathbb{R}^{4}}|\nabla\eta|^{2}dx.
\end{align*}
Hence,
\begin{equation}
\int_{\mathbb{R}^{4}}|\nabla u|^{2}\eta^{2}dx\leq\frac{4}{\rho^{2}}\int_{\mathbb{R}^{4}}|\nabla\eta|^{2}dx<\frac{4C_{2}}{\rho^{2}}R^{2}.\label{eq:prep integral estimate}
\end{equation}
By (\ref{eq:cutoff eta gradient estimate}) and (\ref{eq:prep integral estimate}),
\begin{align*}
\int_{\mathbb{R}^{4}}\left|\nabla_{kj}^{2}\eta^{4}u_{k}u_{j}\right|dx & \leq C_{3}\int_{B_{2R}(0)}\left(\eta|\nabla^{2}\eta|+|\nabla\eta|^{2}\right)\eta^{2}|\nabla u|^{2}dx\\
 & \le C_{4}\frac{1}{R^{2}}\int_{\mathbb{R}^{4}}\eta^{2}|\nabla u|^{2}dx\\
 & \leq\frac{C_{5}}{\rho^{2}}.
\end{align*}
Similarly, 
\begin{align*}
\int_{\mathbb{R}^{4}}|\Delta\eta^{4}|\cdot|\nabla u|^{2}dx & \leq\frac{4C_{6}}{\rho^{2}}.
\end{align*}
Hence we have proved this lemma.
\end{proof}
Combining Lemmas \ref{lem: computation of int sigm_2} and \ref{lem: integral gradient estimate  },
we prove Theorem \ref{thm:Integral finit}.

\begin{proof}
[Proof of Theorem \ref{thm:Integral finit}]By Lemma \ref{lem:divergence },
${\rm div}(|\nabla u|^{2}\nabla u)<0$. From Lemma \ref{lem: computation of int sigm_2},
since $\rho>0,$
\begin{align*}
2\int_{\mathbb{R}^{4}}\sigma_{2}(A(\rho,u))\eta^{4}dx & =\int_{\mathbb{R}^{4}}\left(\left(\nabla_{kj}^{2}\eta^{4}\right)u_{k}u_{j}-|\nabla u|^{2}\Delta\eta^{4}\right)dx\\
 & \quad+\int_{\mathbb{R}^{4}}\rho\eta^{4}{\rm div}(|\nabla u|^{2}\nabla u)dx\\
 & \leq\int_{\mathbb{R}^{4}}\left(\left(\nabla_{kj}^{2}\eta^{4}\right)u_{k}u_{j}-|\nabla u|^{2}\Delta\eta^{4}\right)dx.
\end{align*}
By Lemma \ref{lem: integral gradient estimate  },
\begin{align}
2\int_{\mathbb{R}^{4}}\sigma_{2}(A(\rho,u))\eta^{4}dx & \leq\frac{C_{1}}{\rho^{2}}.\label{eq:integral of sigma_2}
\end{align}
As the right-hand side of (\ref{eq:integral of sigma_2}) is independent
of $R$, let $R\to\infty$, it holds
\[
\int_{\mathbb{R}^{4}}\sigma_{2}(A(\rho,u))dx<\frac{C_{1}}{2\rho^{2}}<\infty.
\]
We have finished the proof.
\end{proof}

\section{Local estimates\label{sec:local-estimates}}

In this section, following the work of Chen \cite{Chen2005LocalEF}
and Guan-Wang \cite{guanWan2003}, we establish local estimates for
the equation (\ref{eq:local estimate equation}). 

Let $\Omega\subset\mathbb{R}^{4}$ be an open domain. Consider the
equation in $\Omega$:
\begin{equation}
\begin{cases}
\sigma_{2}(A(\rho,u))=f(u)=\exp(4u)p(u),\\
\sigma_{1}(A(\rho,u))>0.
\end{cases}\label{eq:local estimate equation}
\end{equation}
Here $p:\mathbb{R\to\mathbb{R}}$ is a smooth positive polynomial-like
function such that for some $m\geq0$, $C_{p}>0$, and any $a\in\mathbb{R}$,
it holds
\begin{equation}
p(a)+\left|p'(a)\right|(1+|a|)+\left|p''(a)\right|(1+|a|)^{2}\leq C_{p}\left(1+|a|\right)^{4m},\label{eq:growth condition of P(u)}
\end{equation}
and
\begin{equation}
p(a)>C_{p}^{-1}(1+|a|)^{4m}.\label{eq:growth condition 2}
\end{equation}
In particular, $p$ can be a positive constant, or any positive polynomial.

We note that comparing to (\ref{eq:PDE system}), we have fixed the
positive constant $\beta$ as $4$ in (\ref{eq:local estimate equation}).
This does not however reduce the generality due to Remark \ref{rem:+/-cone}.

First, we state the following local $C^{1}$ and $C^{2}$ estimate,
which is a key ingredient for (\ref{eq:local estimate equation}).
\begin{thm}
[\cite{Chen2005LocalEF}, Theorem 1.1 ]\label{thm:Gradient and second order estimate}Let
$u$ be a $C^{4}$-solution to (\ref{eq:local estimate equation})
for $\rho>0$ in $B_{2R}(0)$. Assume that $p$ satisfies (\ref{eq:growth condition of P(u)}).
Then, we have for $x\in B_{R}(0)$, 
\begin{equation}
\rho|\nabla^{2}u|+\rho^{2}|\nabla u|^{2}\le C\left(\frac{1}{|R|^{2}}+c_{sup}(f)\right),\label{eq:C^2 estimate by Chen}
\end{equation}
for some universal constant $C$, and 
\[
c_{sup}(f)=\sup_{x\in B_{R}}\left(|f(u)|^{\frac{1}{2}}+|\left(f^{\frac{1}{2}}\right)'|+|\left(f^{\frac{1}{2}}\right)''|\right).
\]
\end{thm}

\begin{rem}
If we start with a $C^{2}$ solution $u$ of (\ref{eq:local estimate equation})
in $B_{R}(0)$, by our assumption of $f(u)$ in (\ref{eq:local estimate equation}),
and the fact that $\sigma_{2}^{\frac{1}{2}}$ is concave in $\Gamma_{2}^{+}$,
we may apply Evans-Krylov theorem and linear elliptic theory to conclude
that $u\in C^{4}(B_{R/2}(0))$. See \cite{Caffarelli1995FullyNE},
chapter 8 and 9. In particular, by Chen's estimate (\ref{eq:C^2 estimate by Chen}),
all higher order estimates only depend on $R$, $\sup|f(u)|^{\frac{1}{2}}$,
and bounds of derivatives of $f^{\frac{1}{2}}$. Without no confusion,
the solutions in the following sections are at least $C^{4}$ smooth.
\end{rem}

Second, we establish the following improved estimate, which is a modification
of Proposition 3.6 in \cite{guanWan2003} to our settings, also c.f
\cite{Hhan2004local,LiWangLuc-ajm}. From now on, we fix a $\rho>0.$
\begin{thm}
\label{thm:Asyp upperbound Guan Wang } Assume that $p$ satisfies
(\ref{eq:growth condition of P(u)}) and (\ref{eq:growth condition 2}).
Let $u$ be a solution to (\ref{eq:local estimate equation}) for
$\rho>0$ in $B_{R}(0)$ for some $R>1$. For any $\delta>0$, there
exists an $\epsilon_{0}=\epsilon_{0}(\delta,p,m,\rho)>0$ independent
of $R$ such that if 
\[
\int_{B_{R}(0)}f(u)dx<\epsilon_{0},
\]
 then
\begin{equation}
\sup_{B_{\frac{1}{2}R}(0)}\frac{1}{4}\ln f(u)+\ln R<\ln\delta.\label{eq:asymp estimate}
\end{equation}
\end{thm}

Using Theorems \ref{thm:Gradient and second order estimate} and
\ref{thm:Asyp upperbound Guan Wang }, we obtain the key estimates:
\begin{cor}
\label{cor:Guan-Wang}Let $u$ be a solution to (\ref{eq:PDE system})
for $\rho>0$. Assume that $p$ satisfies (\ref{eq:growth condition of P(u)})
and (\ref{eq:growth condition 2}). Then,
\begin{equation}
\lim_{x\to\infty}u(x)+\ln|x|=-\infty.\label{eq:asyp direct from Guan-Wang}
\end{equation}
 For $|x|>0$ , there is a universal constant $c$ such that 
\begin{align}
|\nabla u| & \leq\frac{c}{\rho|x|},\ |\nabla^{2}u|\leq\frac{c}{\rho|x|^{2}}.\label{eq:estimate actually used-1}
\end{align}
\end{cor}

\begin{proof}
Fix $\delta>0$ and $\epsilon=\epsilon_{0}(\delta)$ in Theorem \ref{thm:Asyp upperbound Guan Wang }.
\textcolor{black}{By Theorem \ref{thm:Integral finit}, a solution
to (\ref{eq:PDE system}) satisfies $\int_{\mathbb{R}^{4}}f(u)<\infty$.}
For $\epsilon$, we can find a big $\tilde{R}$ such that $\int_{\mathbb{R}^{4}\backslash B_{\tilde{R}}(0)}f(u)dx<\epsilon$.
Then, for any $x\in\mathbb{R}^{4}\backslash B_{2\tilde{R}}(0)$, we
have $B_{\frac{|x|}{2}}(x)\subset\mathbb{R}^{4}\backslash B_{\tilde{R}}(0)$,
and thereby 
\[
\sup_{B_{\frac{1}{2}|x|}(x)}\frac{1}{4}\ln f(u)<-\ln\frac{|x|}{2}+\ln\delta,
\]
which implies (\ref{eq:asyp direct from Guan-Wang}). By (\ref{eq:C^2 estimate by Chen}),
\[
\rho|\nabla^{2}u|+\rho^{2}|\nabla u|^{2}\leq\frac{c}{|x|^{2}}+\frac{4c}{|x|^{2}},
\]
implying (\ref{eq:estimate actually used-1}) for $|x|>2\tilde{R}$
. We obtain (\ref{eq:estimate actually used-1}) in $B_{2\tilde{R}}(0)\backslash\{0\}$
with a proper choice of constant $c$. 
\end{proof}
In the rest of this section, we prove Theorem \ref{thm:Asyp upperbound Guan Wang }.
We begin with a technical lemma.
\begin{lem}
\label{lem: elementrary}Let $p:\mathbb{R}\to\mathbb{R}^{+}$ be a
smooth positive function.
\begin{enumerate}
\item If $p$ satisfies (\ref{eq:growth condition of P(u)}) and for some
$b$
\begin{equation}
\frac{e^{4a}p(a+b)}{\left(1+|b|\right)^{4m}}>c_{0}>0,\label{eq:lowerbounde}
\end{equation}
then $a>C_{0}(c_{0},p,m)$ independent of $b$.
\item If $p$ satisfies (\ref{eq:growth condition 2}) and for some $b$
\begin{equation}
\frac{e^{4a}p(a+b)}{\left(1+|b|\right)^{4m}}<c_{1},\label{eq:upperbounde}
\end{equation}
 then $a<C_{1}(c_{1},p,m)$ independent of $b$.
\end{enumerate}
\end{lem}

\begin{proof}
If $p$ satisfies (\ref{eq:growth condition of P(u)}) and (\ref{eq:lowerbounde}),
\begin{align*}
c_{0} & <\frac{e^{4a}p(a+b)}{\left(1+|b|\right)^{4m}}<\frac{e^{4a}C_{p}(1+|a+b|)^{4m}}{\left(1+|b|\right)^{4m}}\\
 & \leq e^{4a}C_{2}(1+|a|)^{4m},
\end{align*}
for some positive $C_{2}=C_{2}(p,m)$. Thus, 
\[
e^{4a}(1+|a|)^{4m}\ge\frac{c_{0}}{C_{2}}.
\]
Hence, $a$ is bounded from below by some properly chosen constant
$C_{0}(c_{0},p,m)$. 

Next, we prove the second statement. Suppose that $p$ satisfies (\ref{eq:growth condition 2})
and (\ref{eq:upperbounde}). If $a<0,$ the statement already holds.
We may assume that $a>0$. Note by (\ref{eq:growth condition 2}),
\[
C_{p}^{-1}e^{4a}\frac{1}{\left(1+|b|\right)^{4m}}<\frac{e^{4a}p(a+b)}{\left(1+|b|\right)^{4m}}<c_{1}.
\]
Thus, 
\begin{equation}
a<\frac{1}{4}\ln\left(c_{1}C_{p}\right)+m\ln(1+|b|).\label{eq:eq1}
\end{equation}
Pick a positive constant $C_{3}=C_{3}(c_{1},p,m)$ so that 
\[
\frac{1}{2}|b|>\frac{1}{4}\ln\left(c_{1}C_{p}\right)+m\ln(1+|b|)
\]
 whenever $|b|>C_{3}$. We discuss the following 3 cases.

If $|b|\leq C_{3}$, then by (\ref{eq:eq1}), $a\leq\frac{1}{4}\ln\left(c_{1}C_{p}\right)+m\ln(1+C_{3})$. 

If $b>0$, by (\ref{eq:growth condition 2}) and the fact that $a>0,$
\[
C_{p}^{-1}e^{4a}<C_{p}^{-1}e^{4a}\frac{(1+|a+b|)^{4m}}{\left(1+|b|\right)^{4m}}<c_{1}.
\]
Thus, we get $a<\frac{1}{4}\ln\left(c_{1}C_{p}\right)$. 

If $b<-C_{3}$, then 
\begin{align*}
-b-a & \ge|b|-\frac{1}{4}\ln\left(c_{1}C_{p}\right)-m\ln(1+|b|)\\
 & >\frac{|b|}{2}>0.
\end{align*}
Thus, by (\ref{eq:growth condition 2}),
\begin{align*}
C_{p}^{-1}\left(\frac{1}{2}\right)^{4m}e^{4a} & \leq e^{4a}C_{p}^{-1}\frac{\left(1+\frac{|b|}{2}\right)^{4m}}{\left(1+|b|\right)^{4m}}<\frac{e^{4a}p(a+b)}{\left(1+|b|\right)^{4m}}<c_{1}.
\end{align*}
Hence, 
\[
a<m\ln2+\frac{1}{4}\ln\left(c_{1}C_{p}\right).
\]

Combining all 3 cases, we have finished the proof.
\end{proof}
Finally, we are ready to prove Theorem \ref{thm:Asyp upperbound Guan Wang }.
\begin{proof}
(of Theorem \ref{thm:Asyp upperbound Guan Wang }). Let $\delta>0$.
Suppose that
\begin{equation}
\sup_{B_{\frac{R}{2}}(0)}\left(\ln f(u)+4\ln R\right)>4\ln\delta.\label{enu:2}
\end{equation}
We only need to show that
\begin{equation}
\int_{B_{R}(0)}f(u)dx>\epsilon_{0},\label{enu:1}
\end{equation}
for some $\epsilon_{0}=\epsilon_{0}(\delta,p,m,\rho)$. Define 
\begin{equation}
\phi:\lambda\mapsto\left(\frac{3}{4}-\lambda\right)^{4}\sup_{B_{\lambda R}(0)}f(u)\label{eq:defninition of f}
\end{equation}
for $\lambda\in(0,\frac{3}{4})$. Then, there exit $z\in B_{3R/4}(0)$
and $\lambda_{0}\in(0,3/4)$ such that 
\begin{equation}
\phi(\lambda_{0})=\sup_{\lambda_{0}\in(0,\frac{3}{4})}\phi(\lambda_{0}),\quad f(u(z))=\sup_{x\in B_{\lambda_{0}R}(0)}f(u(x)).\label{eq:lambda_i andz_i}
\end{equation}
Pick $s=\frac{1}{2}\left(\frac{3}{4}-\lambda_{0}\right)$. We have
$B_{sR}(z)\subset B_{R}(0)$ and
\[
\sup_{B_{sR}(z)}f(u)\leq\sup_{B_{\left(\lambda_{0}+s\right)R}(0)}f(u).
\]
Denote 
\[
\mu:=f(u(z))^{\frac{1}{4}}.
\]
Since 
\begin{align*}
\left(\frac{3}{4}-s-\lambda_{0}\right)^{4}\sup_{B_{\left(s+\lambda_{0}\right)R}(0)}f(u) & \leq\left(\frac{3}{4}-\lambda_{0}\right)^{4}\sup_{B_{\lambda_{0}R}(0)}f(u)\\
 & =\left(\frac{3}{4}-\lambda_{0}\right)^{4}\mu^{4},
\end{align*}
\begin{equation}
\sup_{B_{sR}(z)}f(u)\leq\sup_{B_{\left(\lambda_{0}+s\right)R}(0)}f(u)\leq2^{4}\mu^{4}.\label{eq:Guan-Wang estimate (harnarck)-1}
\end{equation}
Define the rescaled function $v$ as
\[
v(x):=u(z+\mu^{-1}x)-\ln\mu',
\]
where $\mu'$ satisfies 
\[
\mu'(1+|\ln\mu'|)^{m}=\mu.
\]
Let $\tilde{p}(v):=\frac{p(v+\ln\mu')}{(1+|\ln\mu'|)^{4m}}.$ Then,
$v$ satisfies
\[
\sigma_{2}(A(\rho,v))=\frac{f(v+\ln\mu')}{(1+|\ln\mu'|)^{4m}}=e^{4v}\tilde{p}(v).
\]

From (\ref{enu:2}), we have 
\begin{align*}
\phi\left(\frac{1}{2}\right)R^{4} & =\left(\frac{1}{4}\right)^{4}\sup_{B_{\frac{R}{2}}(0)}e^{\ln f(u)+4\ln R}>(\frac{1}{4})^{4}\delta^{4}.
\end{align*}
Then, by (\ref{eq:defninition of f}) and (\ref{eq:lambda_i andz_i}),
\begin{align*}
s^{4}\mu^{4}R^{4} & =\left(\frac{\frac{3}{4}-\lambda_{0}}{2}\right)^{4}f(u(z))R^{4}\\
 & \geq2^{-4}\phi(\lambda_{0})R^{4}\\
 & \geq2^{-4}\phi\left(\frac{1}{2}\right)R^{4}\\
 & >2^{-4}4^{-4}\delta^{4}.
\end{align*}
Hence, $B_{s\mu R}(0)$ contains a fixed ball $B_{R'}(0)$, if we
denote $R':=8^{-1}\delta$. 

We claim that $v$ is uniformly bounded in $B_{R'}(0)$. Since
\begin{align*}
1 & =e^{4v(0)}\tilde{p}(v(0))=e^{4v(0)}\frac{p(v(0)+\ln\mu')}{(1+|\ln\mu'|)^{4m}},
\end{align*}
we have $v(0)>C_{0}(1,p,m)$ by Lemma \ref{lem: elementrary}. By
(\ref{eq:Guan-Wang estimate (harnarck)-1}),
\[
\sup_{x\in B_{s\mu R}(0)}\frac{e^{4v}p(v+\ln\mu')}{\left(1+|\ln\mu'|\right)^{4m}}=\sup_{x\in B_{sR}(z)}\frac{f(u)}{\mu^{4}}\leq2^{4},
\]
and by Lemma \ref{lem: elementrary}, we have $\sup_{B_{s\mu R}(0)}v<C_{1}(2^{4},p,m).$
Since $\tilde{p}(v)$ also satisfies growth condition (\ref{eq:growth condition of P(u)}),
and $v$ is bounded from above by $C_{1}(2^{4},p,m)$, it holds that
\begin{align}
c_{sup}(e^{2v}\tilde{p}^{\frac{1}{2}}(v)) & \leq C_{4}e^{2v}\left(\tilde{p}^{\frac{1}{2}}(v)+|(\tilde{p}^{\frac{1}{2}}(v))'|+|(\tilde{p}^{\frac{1}{2}}(v))''|\right)\label{eq:bounded ness of p_i}\\
 & \leq C_{5}e^{2v}\left(1+|v|\right)^{2m}\nonumber \\
 & <C_{6},\nonumber 
\end{align}
where $C_{4},C_{5},C_{6}$ are constants only depending on $p$ and
$m$. From (\ref{eq:bounded ness of p_i}) and Theorem \ref{thm:Gradient and second order estimate},
\begin{equation}
\sup_{B_{R'}(0)}|\nabla v|^{2}<\frac{C}{\rho^{2}}\left(\frac{1}{\left(R'\right)^{2}}+c_{sup}(e^{2v}\tilde{p}^{\frac{1}{2}}(v))\right)<C_{7},\label{eq:gradientestimate in Guan-Wang}
\end{equation}
for some constant $C_{7}=C_{7}(R',C_{6},p,m,\rho)=C_{7}(p,m,\rho)$.
The gradient estimate (\ref{eq:gradientestimate in Guan-Wang}) together
with the fact that $C_{0}(1,p,m)<v(0)\leq C_{1}(2^{4},p,m)$ implies
that $v$ is uniformly bounded in $B_{R'}(0)$. 

However, if $v$ is uniformly bounded in $B_{R'}(0)$, then 
\[
\int_{B_{R'}(0)}e^{4v}\tilde{p}(v)dx>C_{8}(\delta,p,m,\rho)>0.
\]
Let $\epsilon_{0}=C_{8}$ and we obtain (\ref{enu:1}). Thus, we have
finished the proof.
\end{proof}

\section{Asymptotic behavior \label{sec:Asymptotic-behavior}}

In this section, we use the existing estimate of Corollary \ref{cor:Guan-Wang}
to describe the asymptotic behavior of our solution $u$ near infinity.
As a consequence, we establish the crucial fact that the shape of
level sets of $u$, after some proper re-scaling, is asymptotically
spherical. The main result in this section is Theorem \ref{thm:convergence of level set}. 

First, we set some notations for our analysis. Throughout this section,
$u$ is a solution of (\ref{eq:PDE system}) with $\beta=4$ and \textcolor{black}{$\rho>0$.}
\begin{defn}
We define
\[
\bar{u}(r):=\max_{|x|=r}u(x),\ \underline{u}(r):=\min_{|x|=r}u(x);
\]
\[
\overline{r}(t):=\max\{|x|:u(x)=t\},\ \underline{r}(t):=\min\{|x|:u(x)=t\}.
\]
\end{defn}

Second, we list some simple facts.
\begin{lem}
For some constant $C=C(\rho),$ we have
\begin{equation}
\bar{u}(r)-\underline{u}(r)\leq C.\label{eq:add20}
\end{equation}
\end{lem}

\begin{proof}
This is a direct consequence of the gradient estimate (\ref{eq:estimate actually used-1}). 
\end{proof}
\begin{lem}
\label{rem:decreasing}$\underline{u}(r)$ is non-increasing. 
\end{lem}

\begin{proof}
Since $u\in\mathcal{C}_{2}^{+}$,
\[
\sigma_{1}(A(\rho,u))=-\Delta u-\rho|\nabla u|^{2}>0.
\]
We verify the claim by the maximum principle.
\end{proof}
\begin{lem}
We have
\begin{equation}
\underline{u}(\underline{r}(t))=t,\ \overline{u}(\overline{r}(t))=t,\ \underline{r}(\underline{u}(r))=r,\ \overline{r}(\overline{u}(r))=r.\label{eq:underline=00007Bu=00007D(underline=00007Br=00007D(t))=00003Dt}
\end{equation}
\end{lem}

\begin{proof}
For a given $t$, there exists $x_{0}$ such that $|x_{0}|=\underline{r}(t)$
and $u(x_{0})=t$. Then,
\[
\underline{u}(\underline{r}(t))\leq u(x_{0})=t.
\]
By the maximum principle,
\[
\underline{r}(t)=\min\{|x|:u(x)=t\}=\min\{|x|:u(x)\leq t\}.
\]
Hence, $u\geq t$ in $\overline{B_{\underline{r}(t)}(0)}$. Therefore,
$\underline{u}(\underline{r}(t))=t$. 

The proof for other identities (\ref{eq:underline=00007Bu=00007D(underline=00007Br=00007D(t))=00003Dt})
is similar so we omit it here.
\end{proof}
Third, we consider the growth rate of $\underline{u}(r)$ and $u(x).$ 
\begin{lem}
Let $s=\log r$. Denote $\frac{d}{ds}\left(\underline{u}(e^{s})\right)^{+}$
to be the right derivative with respect to $s$. Then,
\begin{equation}
\lim_{s\to\infty}\frac{d}{ds}\left(\underline{u}(e^{s})\right)^{+}\geq-\frac{2}{\rho}.\label{eq:new1}
\end{equation}
\label{lem:.alpha >-2/rho}
\end{lem}

\begin{proof}
By Lemma \ref{rem:decreasing} and Theorem \ref{lem:Li-Nguyen Monotonicity formula },
$\underline{u}$ is non-increasing and concave as a function of $s$.
Thus, the right derivative of $\underline{u}(e^{s})$ exists and is
monotonic. By the classic Alexandrov's theorem, the first and second
derivatives $\underline{u}_{s}$, $\underline{u}_{ss}$ exist almost
everywhere and 
\begin{equation}
\underline{u}_{ss}\leq0\quad a.e.\label{eq:30}
\end{equation}
In particular, the limit appeared in (\ref{eq:new1}) exists. 

By Definition \ref{def:Viscosity definition}, for almost all $s,$
\begin{equation}
3\underline{u}_{ss}\left(\underline{u}_{s}+\frac{\rho}{2}\underline{u}_{s}^{2}\right)>0.\label{eq:radial v supersolution}
\end{equation}
We combine (\ref{eq:30}) and (\ref{eq:radial v supersolution}) to
prove (\ref{eq:new1}).
\end{proof}
\begin{lem}
\label{lem:alpha}If $u$ is a solution to (\ref{eq:PDE system}),
then there exists $\alpha\in[-\frac{2}{\rho},-1)$ such that 
\[
\lim_{|x|\to\infty}\frac{u(x)}{\ln|x|}=\alpha.
\]
\end{lem}

\begin{proof}
Define 
\[
\alpha:=\lim_{s\to\infty}\frac{\underline{u}(e^{s})}{s}.
\]
By the concavity of $\underline{u}(e^{s})$ and Lemma \ref{lem:.alpha >-2/rho},
$\alpha$ is well defined and is bounded below by $-\frac{2}{\rho}$.
By Corollary \ref{cor:Guan-Wang}, $\alpha\leq-1$. We claim that
$\alpha<-1$. 

If not, then we assume that $\alpha=-1.$ We use Theorem \ref{lem:Li-Nguyen Monotonicity formula }
to argue that for $s_{0}<s_{1}$,
\[
\frac{\underline{u}(e^{s_{1}})-\underline{u}(e^{s_{0}})}{s_{1}-s_{0}}\geq\lim_{s\to\infty}\frac{\underline{u}(e^{s})-\underline{u}(e^{s_{0}})}{s-s_{0}}=\alpha,
\]
\begin{equation}
0\leq\frac{\underline{u}(e^{s_{1}})-\underline{u}(e^{s_{0}})}{s_{1}-s_{0}}-\alpha=\frac{\underline{u}(e^{s_{1}})-\alpha s_{1}-(\underline{u}(e^{s_{0}})-\alpha s_{0})}{s_{1}-s_{0}}.\label{eq:add4}
\end{equation}
Hence, $\underline{u}(e^{s})-\alpha s$ is increasing. For $s>0,$
\begin{equation}
\underline{u}(e^{s})-\alpha s\geq\underline{u}(1)-0>-\infty,\label{eq:Leamma alpha contradiction to Guan-Wang}
\end{equation}
which contradicts with (\ref{eq:asyp direct from Guan-Wang}) since
$\alpha=-1.$ We have thus proved that $\alpha<-1$.

In addition, by (\ref{eq:add20}),
\[
\lim_{r\to\infty}\frac{\bar{u}(r)}{\ln r}=\lim_{r\to\infty}\frac{\underline{u}(r)}{\ln r}=\alpha.
\]
Since $\underline{u}(|x|)\leq u(x)\leq\bar{u}(|x|)$, we have thus
proved the lemma.
\end{proof}
An immediate consequence of Lemma \ref{lem:.alpha >-2/rho} and Lemma
\ref{lem:alpha} is the following non-existence result for $\rho\geq2$.
\begin{cor}
\textcolor{black}{If $\rho\geq2$, then there is no solution to (\ref{eq:PDE system}).}\textcolor{red}{\label{cor:If rho>=00003D2}}
\end{cor}

Fourth, we use Lemma \ref{lem:alpha} to control the shape of level
sets of $u$. The next proposition shows that $\overline{r}$ and
$\underline{r}$ cannot grow disproportionately as $t\to-\infty$. 
\begin{prop}
There exists some uniform constant $C=C(\rho)$ such that
\[
\limsup_{t\to-\infty}\frac{\bar{r}(t)}{\underline{r}(t)}<C.
\]
\label{prop:level set bounded}
\end{prop}

\begin{proof}
For any $t,$ define $t'=\underline{u}(\bar{r}(t))<t$. By (\ref{eq:underline=00007Bu=00007D(underline=00007Br=00007D(t))=00003Dt}),
$\underline{r}(t')=\overline{r}(t)$ and $t=\bar{u}(\bar{r}(t))$.
Thus, by (\ref{eq:add20}), 
\begin{equation}
|t'-t|=|\underline{u}(\overline{r}(t))-\bar{u}(\bar{r}(t))|\leq C.\label{eq:add1}
\end{equation}

We prove the proposition by contradiction. If there is a sequence
$\{t_{i}\}$ such that $\lim_{i\to\infty}\frac{\overline{r}(t_{i})}{\underline{r}(t_{i})}=\infty,$
then with (\ref{eq:add1}), and the obvious fact that $\underline{r}(t_{i})\to\infty$
as $i\to\infty,$ we get
\begin{equation}
\lim_{i\to\infty}\frac{t'_{i}-t_{i}}{\ln\underline{r}(t'_{i})-\ln\underline{r}(t_{i})}=\lim_{i\to\infty}\frac{t'_{i}-t_{i}}{\ln\bar{r}(t{}_{i})-\ln\underline{r}(t_{i})}=0,\label{eq:add10}
\end{equation}
where $t_{i}'=\underline{u}(\bar{r}(t_{i}))$. Let $r_{i}=\underline{r}(t_{i})$
and $r_{i}'=\underline{r}(t'_{i})$. Since $\lim_{r\to\infty}\frac{\underline{u}(r)}{\ln r}=\alpha<-1$
by Lemma \ref{lem:alpha}, replacing by a subsequence, we may assume
that 
\[
\lim_{i\to\infty}\frac{t_{i}-t_{i-1}}{\ln r_{i}-\ln r_{i-1}}=\alpha.
\]
Since $t_{i}'=\underline{u}(r_{i}')$, $r_{i}'=\overline{r}(t_{i})=\underline{r}(t_{i}')$,
after further replacing by a subsequence, we may assume:
\[
r_{1}<r_{1}'<r_{2}<r_{2}'\cdots<r_{n}<r_{n}'<\cdots,
\]
and 
\[
t_{1}>t_{1}'>t_{2}>t_{2}'>\cdots>t_{n}>t_{n}'>\cdots.
\]
By Theorem \ref{lem:Li-Nguyen Monotonicity formula }, we have
\begin{equation}
\frac{t_{i}-t_{i-1}}{\ln r_{i}-\ln r_{i-1}}\geq\frac{t_{i}'-t_{i}}{\ln r_{i}'-\ln r_{i}}.\label{eq:add11}
\end{equation}
However, by (\ref{eq:add10}) and (\ref{eq:add11})
\[
-1>\alpha=\lim_{i\to\infty}\frac{t_{i}-t_{i-1}}{\ln r_{i}-\ln r_{i-1}}\geq\lim_{i\to\infty}\frac{t_{i}'-t_{i}}{\ln r_{i}'-\ln r_{i}}=0,
\]
which is a contradiction. We have thus established our claim. 
\end{proof}
Finally, we present our blow-down analysis. Consider a monotone decreasing
sequence $\{t_{i}\}$ such that $t_{i}\to-\infty$ and the level sets
\[
L_{t_{i}}=\{x:u(x)=t_{i}\}.
\]
Define the following blow-down sequence of functions 
\begin{equation}
u_{i}(x):=u(\underline{r}(t_{i})x)-t_{i}.\label{eq:definition}
\end{equation}
The re-scaled level sets are defined by 
\[
\tilde{L}_{i}:=\frac{L_{t_{i}}}{\underline{r}(t_{i})}=\{x:u_{i}(x)=0\}.
\]
We present the main result in this section.

\begin{thm}
\label{thm:convergence of level set}For any compact subset $E\subset\mathbb{R}^{4}\backslash\{0\}$,
after replacing by a subsequence,
\[
\lim_{i\to\infty}\|u_{i}-\alpha\ln|x|\|_{C^{1,\gamma}(E)}=0,
\]
 for any $\gamma\in(0,1).$
\end{thm}

\begin{proof}
By the definition of $\underline{r}(t_{i})$, we have
\begin{equation}
\inf_{B_{1}\backslash\{0\}}u_{i}=\min_{\partial B_{1}(0)}u_{i}=0.\label{eq:add2}
\end{equation}
From (\ref{eq:estimate actually used-1}), for some uniform constant
$C=C(\rho)$, 
\[
|\nabla u_{i}(x)|=\underline{r}(t_{i})|\nabla u(\underline{r}(t_{i})x)|\leq\underline{r}(t_{i})\left(\frac{C}{\underline{r}(t_{i})|x|}\right)\leq\frac{C}{|x|},
\]
\[
|\nabla^{2}u_{i}(x)|=\underline{r}^{2}(t_{i})|\nabla^{2}u(\underline{r}(t_{i})x)|\leq\underline{r}(t_{i})^{2}\left(\frac{C}{\underline{r}(t_{i})^{2}|x|^{2}}\right)\leq\frac{C}{|x|^{2}}.
\]
Thus, $u_{i}$ is uniformly bounded in any compact set $E\subset\mathbb{R}^{4}\backslash\{0\}$
in $C^{2}(E)$ norm. By Arzela-Ascoli theorem, $u_{i}\to u_{\infty}$
subsequentially in $C^{1,\gamma}(E)$ for any $\gamma\in(0,1)$ and
$u_{\infty}\in C^{1,1}(\mathbb{R}^{4}\backslash\{0\})$.

Notice that $u_{i}$ satisfies the equation 
\[
\sigma_{2}(A(\rho,u_{i}))=\left(\underline{r}(t_{i})\right)^{4}e^{4t_{i}}p(u_{i}+t_{i})e^{4u_{i}}.
\]
By Lemma \ref{lem:alpha},
\[
\left(\underline{r}(t_{i})\right)^{4}e^{4t_{i}}e^{4u_{i}}p(u_{i}+t_{i})\to0,
\]
uniformly in any compact subsets of $\mathbb{R}^{4}\backslash\{0\}$.
Therefore, by $C_{loc}^{1,\gamma}$ convergence, $u_{\infty}$ satisfies
$A(\rho,u_{\infty})\in\partial\Gamma_{2}^{+}$ in viscosity sense
(see \cite{Crandall1992UsersGT} section 6 , \cite{Caffarelli1995FullyNE}
chapter 2). Since $A(\rho,u_{\infty})=\rho A(1,\frac{1}{\rho}u_{\infty})$,
from Theorem 1.18 in \cite{li2009local} or Theorem 1.5 in \cite{Li-ARMA06},
\begin{equation}
u_{\infty}=c\ln|x|+c',\label{eq:add3}
\end{equation}
where $c,c'$ are constants. In particular, by (\ref{eq:add2}), $\min_{|x|=1}u_{\infty}(x)=0$,
and then $c'=0$.

Next, we prove that $u_{\infty}(x)=\alpha\ln|x|$. Fix $x$ such that
$|x|>1$. By Theorem \ref{lem:Li-Nguyen Monotonicity formula }, the
concavity of $\underline{u}\circ\exp$ implies that 
\[
\underline{u}(r|x|)\frac{\ln r-\ln1}{\ln(r|x|)-\ln1}+\underline{u}(1)\frac{\ln(r|x|)-\ln r}{\ln(r|x|)-\ln1}\leq\underline{u}(r).
\]
Hence,
\begin{equation}
\frac{\underline{u}(r|x|)-\underline{u}(r)}{\ln(r|x|)-\ln r}\leq\frac{\underline{u}(r|x|)-\underline{u}(1)}{\ln(r|x|)-\ln1}.\label{eq:add23}
\end{equation}
By Lemma \ref{lem:alpha},
\begin{equation}
\lim_{r\to\infty}\frac{\underline{u}(r|x|)-\underline{u}(1)}{\ln(r|x|)-\ln1}=\alpha.\label{eq:new2}
\end{equation}
Thus, for any $\epsilon>0,$ by (\ref{eq:add23}) and (\ref{eq:new2}),
there exists a $N\in\mathcal{\mathbb{N}}$ such that for all $i>N,$
\[
\underline{u}(\underline{r}(t_{i})|x|)-\underline{u}(r)<(\alpha+\epsilon)\ln(|x|).
\]
Noticing that $t_{i}=\underline{u}(\underline{r}(t_{i}))$ by (\ref{eq:underline=00007Bu=00007D(underline=00007Br=00007D(t))=00003Dt}),
we use (\ref{eq:add20}) to conclude that for any $i>N,$
\begin{align}
u_{i}(x) & =u(\underline{r}(t_{i})x)-t_{i}\label{eq:add24}\\
 & \leq\underline{u}(\underline{r}(t_{i})|x|)+C-\underline{u}(\underline{r}(t_{i}))\nonumber \\
 & \leq(\alpha+\epsilon)\ln|x|+C.\nonumber 
\end{align}
Thus, $\limsup_{i\to\infty}\frac{u_{i}(x)}{\ln|x|}\leq\alpha.$ On
the other hand, by (\ref{eq:add4}),
\[
\frac{\underline{u}(r|x|)-\alpha\ln\left(r|x|\right)-\left(\underline{u}(r)-\alpha\ln r\right)}{\ln|x|}\geq0.
\]
Hence, a similar argument shows that $\liminf_{i\to\infty}\frac{u_{i}(x)}{\ln|x|}\geq\alpha.$
We have thus proved that $c=\lim_{i\to\infty}\frac{u_{i}(x)}{\ln|x|}=\alpha.$
\end{proof}
We finish this section by stating the following corollary, which will
be useful later.
\begin{cor}
\label{cor: gradient convergence} For a sequence $t_{i}\to-\infty$,
$u_{i}$ is given as in (\ref{eq:definition}). After passing to a
possible subsequence, we have
\begin{align*}
 & \lim_{i\to\infty}\left\Vert \frac{\nabla u_{i}(y)}{|\nabla u_{i}(y)|}-\frac{y}{|y|}\right\Vert _{L^{\infty}(\tilde{L}_{i})}=0,\\
 & \lim_{i\to\infty}\left\Vert y\cdot\nabla u_{i}(y)-\alpha\right\Vert _{L^{\infty}(\tilde{L}_{i})}=0.
\end{align*}

\end{cor}

\begin{proof}
By Proposition \ref{prop:level set bounded}, we may pick a $R\in\mathbb{R}$
such that $\tilde{L}_{i}=L_{t_{i}}/\underline{r}(t_{i})\subset E=\overline{B_{R}(0)\backslash B_{\frac{1}{R}}(0)}$.
Thus, for any $y\in E,$ $|\nabla\ln|y||=\frac{1}{|y|}\in[1/R,R]$.
By Theorem \ref{thm:convergence of level set}, there exists a large
$N\in\mathbb{N}$ such that for all $i>N,$ and $y\in\tilde{L}_{t_{i}}$
\[
\frac{1}{2R}<|\nabla u_{i}(y)|<\frac{2}{R}.
\]
Thus, by Theorem \ref{thm:convergence of level set} again, for $y\in\tilde{L}_{i}\subset E$
\begin{align*}
\lim_{i\to\infty}\left\Vert \frac{\nabla u_{i}(y)}{|\nabla u_{i}(y)|}-\frac{y}{|y|}\right\Vert _{L^{\infty}(\tilde{L}_{i})} & =0,\\
\lim_{i\to\infty}\left\Vert y\cdot\nabla u_{i}(y)-\alpha\right\Vert _{L^{\infty}(\tilde{L}_{i})} & =0.
\end{align*}
Here the second identity is due to the fact that $y\cdot\nabla\ln|y|=1.$
\end{proof}

\section{Proof of main theorems\label{sec:Proof-of-main theorem}}

In this section, we prove Theorem \ref{thm:Maintheorem}. Using our
Pohozaev identity, the quasi-local mass $M(t_{i})$ can be expressed
by terms defined on $\tilde{L}_{i}$. We argue that $\tilde{L}_{t_{i}}$
sub-converges to the standard round sphere in the Gromov-Hausdorff
sense. Once we find the limits of the related terms on $\tilde{L}_{t_{i}}$,
we can argue that $M(t_{i})$ converges to $0$ as $i\rightarrow\infty$.
Thus, by the monotonicity, $M(t)\equiv0$, and we can establish Theorem
\ref{thm:Maintheorem} using the rigidity result of $M(t)$.

By Remark \ref{rem:+/-cone}, without loss of generality, we may assume
that $\beta=4$ and we use the same notation as in Section 7.

By Theorem \ref{thm:convergence of level set}, Corollary \ref{cor: gradient convergence}
and Proposition \ref{prop:level set bounded}, after passing to a
subsequence, we may assume the following conditions: \renewcommand{\theenumi}{\roman{enumi}} \renewcommand{\labelenumi}{\theenumi)}
\begin{enumerate}
\item \label{enu:For-a-fixed}For some fixed $R>1$, we assume $\tilde{L}_{i}=\{u_{i}=0\}\subset E=\overline{B_{R}(0)\backslash B_{\frac{1}{R}}(0)}$
for all $i\in\mathbb{N};$
\item For some $\gamma\in(0,1)$, and $y\in E,$ 
\begin{align}
\|u_{i}(y)-\alpha\ln|y|\|_{C^{1,\gamma}(E)} & \to0,\label{eq:converge in E sec 8}\\
\lim_{i\to\infty}|y\cdot\nabla u_{i}(y)-\alpha|= & 0;\label{eq:add52}
\end{align}
\item In particular, we assume that for all $i$, 
\begin{equation}
|y\cdot\nabla u_{i}(y)|\geq|\frac{\alpha}{2}|;\label{eq:add51}
\end{equation}
\item We assume that $t_{i}<\underline{u}(1)$.
\end{enumerate}
We set up some notations. We use polar coordinate $(r,\theta)$ for
all $y\in E$, where $r=|y|$ and $\theta=\frac{y}{|y|}\in S^{3}.$
We define the projection $\pi:E\to S^{3}:y\mapsto\frac{y}{|y|},$
and $\pi_{i}=\pi|_{\tilde{L}_{i}}$. Let $dl_{0}$ be the standard
volume form on the unit sphere $S^{3}.$ Let $d\tilde{l}_{i}$ be
the volume element for $\tilde{L}_{i}.$\renewcommand{\theenumi}{\arabic{enumi}} \renewcommand{\labelenumi}{\theenumi .}
\begin{prop}
\label{prop:approximation map}We use notations as listed above. Then,
for any $i$ we have the following claims:
\end{prop}

\begin{enumerate}
\item $\tilde{L}_{i}$ is a regular $C^{2}$ hyper-surface;
\item $\pi_{i}:\tilde{L}_{i}\to S^{3}$ is a $C^{1,\gamma}$ diffeomorphism;
In particular, $\tilde{L}_{i}$ is star-shaped; 
\item \label{enu:condition }$\tilde{L}_{i}$ converges to $S^{3}$ in the
Gromov-Hausdorff sense;
\item For any $\theta\in S^{3},$ we have
\[
\lim_{i\to\infty}((\pi_{i}^{-1})^{*}d\tilde{l}_{i})(\theta)=dl_{0}(\theta).
\]
\end{enumerate}
\begin{proof}
Claim 1 is trivial due to (\ref{eq:add51}). 

For Claim 2, we first prove that $\pi_{i}$ is both surjective and
injective. Since $t_{i}<\underline{u}(1)$, $\tilde{L}_{i}$ is closed,
compact, and enclosing $0$, $\pi_{i}$ is a surjection. If $\pi_{i}$
is not injective, then there exists $\theta_{0}\in S^{3}$ and $0<\psi_{1}<\psi_{2}$
such that $\psi_{1}\theta_{0},\psi_{2}\theta_{0}\in\tilde{L}\subset E$
and $u_{i}(\psi_{1}\theta_{0})=u_{i}(\psi_{2}\theta_{0})=0$. By the
mean value theorem, there exists $\psi_{0}\in(\psi_{1},\psi_{2})$
such that 
\begin{equation}
\nabla u_{i}(\psi_{0}\theta_{0})\cdot\theta_{0}=0.\label{eq:contradiction for not being injective}
\end{equation}
However, by (\ref{eq:add51})
\begin{equation}
|\nabla u_{i}(\psi_{0}\theta_{0})\cdot\theta_{0}|\geq\frac{|\alpha|/2}{\psi_{0}}\geq\frac{|\alpha|}{2R},\label{eq:add59}
\end{equation}
which is contradicted to (\ref{eq:contradiction for not being injective}).
Thus, $\pi_{i}$ is bijective. 

Next, we prove that $\pi_{i}$ is $C^{1,\gamma}$ regular. Let $y=r\theta\in\tilde{L}_{i}$
and 
\[
u_{i}(r\theta)=0.
\]
By (\ref{eq:add51}) and \ref{enu:For-a-fixed}), 
\[
|\frac{\partial u_{i}}{\partial r}(y)|=\frac{|y\cdot\nabla u_{i}|}{|y|}\geq\frac{|\alpha|}{2R}>0.
\]
By the implicit function theorem, we may define $|y|=r_{i}(\theta)$
as a local $C^{1,\gamma}$ function. In addition, in a local coordinate
$(\theta^{1},\theta^{2},\theta^{3})$ on $S^{3}$, 
\begin{equation}
\frac{\partial r_{i}}{\partial\theta^{k}}=-\frac{\partial u_{i}}{\partial\theta^{k}}/\frac{\partial u_{i}}{\partial r}.\label{eq: differential  of partialrover partial theta}
\end{equation}

Since $\pi_{i}$ is a bijective, $r_{i}(\theta)$ is a globally defined
$C^{1,\gamma}$ function, which leads to the conclusion of Claim 2.

To prove Claim 3, it is sufficient to show that
\begin{equation}
\lim_{i\to\infty}\max_{y\in\tilde{L}_{i}}||y|-1|=0.\label{eq:tilde L_i GH convergence.}
\end{equation}
Since $u_{i}(y)=0$ for any $y\in\tilde{L}_{i}$, we get 
\[
|\alpha\ln|y||\leq\|u_{i}-\alpha\ln|y|\|_{C^{1,\gamma}(E)}.
\]
Hence, by the mean value theorem,
\[
\max_{y\in\tilde{L}_{i}}||y|-1|\leq\frac{\|u_{i}(y)-\alpha\ln|y|\|_{C^{1,\gamma}(E)}}{|\alpha|\min_{y\in E}\frac{1}{|y|}}\leq\frac{R}{|\alpha|}\cdot\|u_{i}-\alpha\ln|y|\|_{C^{1,\gamma}(E)},
\]
which leads to (\ref{eq:tilde L_i GH convergence.}). In addition,
by (\ref{eq:add52}) (\ref{eq: differential  of partialrover partial theta}),
we have $\frac{\partial u_{i}(y)}{\partial r}=\frac{y}{|y|}\cdot\nabla u_{i}(y)\to\alpha$
and 
\begin{equation}
\lim_{i\to\infty}\nabla_{\theta}r_{i}(\theta)=0\label{eq:add53}
\end{equation}
uniformly in $C^{\gamma}(S^{3})$ 

For Claim 4, noting that $dl_{0}$ is the standard volume form on
$S^{3}$, we follow the standard volume form computation for star-shaped
hyper-surface to get
\begin{align}
((\pi_{i}^{-1})^{*}d\tilde{l}_{i})(\theta) & =\sqrt{\det\left(r_{i}^{2}{\rm Id}+\nabla_{\theta}r_{i}\otimes\nabla_{\theta}r_{i}\right)}dl_{0}(\theta)\label{eq:add54}\\
 & =r_{i}^{2}\sqrt{r_{i}^{2}+|\nabla_{\theta}r_{i}|^{2}}dl_{0}(\theta).\nonumber 
\end{align}
 We use (\ref{eq:tilde L_i GH convergence.}) (\ref{eq:add53}) and
(\ref{eq:add54}) to prove Claim 4.
\end{proof}
\begin{rem}
For a continuous function $\psi_{i}$ defined on $E$, if $\psi_{i}\to\psi$
uniformly, then $\left(\pi_{i}^{-1}\right)^{*}\psi_{i}\to\psi|_{S^{3}}$
in $C^{0}(S^{3})$. Hence, by Proposition \ref{prop:approximation map},
\begin{equation}
\int_{\tilde{L}_{i}}\psi_{i}d\tilde{l}_{i}=\int_{S^{3}}\left(\pi_{i}^{-1}\right)^{*}\left(\psi_{i}d\tilde{l}_{i}\right)\to\int_{S^{3}}\psi|_{S^{3}}dl_{0}.\label{eq:convergence of measure}
\end{equation}
\end{rem}

Finally, we are ready to prove our main theorem.
\begin{proof}
[Proof of Theorem \ref{thm:Maintheorem}]We use notations as above.
Let $\Omega_{t}=\{x:u(x)>t\}$. By (\ref{eq:MASS new def}), we consider
the sequence
\[
M(t_{i})=2N(t_{i})Q(t_{i})+\frac{\rho}{8}Q^{4}(t_{i})-12P(t_{i}).
\]

Since $Q(t)$ and $N(t)$ are scaling invariant, we have
\[
Q(t_{i})=-\left(\frac{1}{|S^{3}|}\int_{\tilde{L}_{i}}|\nabla u_{i}|^{3}d\tilde{l}_{i}\right)^{\frac{1}{3}},\ N(t_{i})=\frac{1}{2|S^{3}|}\int_{\tilde{L}_{i}}\left(\tilde{H}(y)|\nabla u_{i}|^{2}-\rho|\nabla u_{i}|^{3}\right)d\tilde{l_{i}},
\]
where $\tilde{H}$ is the mean curvature of $\tilde{L}_{i}$. From
Corollary \ref{cor:Pohozaev for K=00003D1} 
\begin{align*}
12P(t_{i}) & =\frac{3}{2|S^{3}|}\int_{L_{t_{i}}}\left(-\frac{3}{4}\rho|\nabla u|^{4}\langle x,\nu\rangle+\frac{2}{3}H|\nabla u|^{3}\langle x,\nu\rangle\right)dl(x)\\
 & =\frac{3}{2|S^{3}|}\int_{\tilde{L}_{i}}\left(-\frac{3}{4}\rho|\nabla u_{i}|^{4}\langle y,\nu\rangle+\frac{2}{3}\tilde{H}|\nabla u_{i}|^{3}\langle y,\nu\rangle\right)d\tilde{l_{i}}(y).
\end{align*}
Thus, 
\begin{align}
M(t_{i})= & \frac{\rho}{8}Q^{4}(t_{i})+\frac{3}{2|S^{3}|}\int_{\tilde{L}_{i}}\left(-\frac{3}{4}\rho|\nabla u_{i}|^{4}\langle y,\nu\rangle+\frac{2}{3}\tilde{H}|\nabla u_{i}|^{3}\langle y,\nu\rangle\right)d\tilde{l_{i}}\label{eq:M(t_i)}\\
 & +2Q(t_{i})\left(\frac{1}{2|S^{3}|}\int_{\tilde{L}_{i}}[\tilde{H}(y)|\nabla u_{i}|^{2}-\rho|\nabla u_{i}|^{3}]d\tilde{l_{i}}(y)\right)\nonumber \\
= & \frac{9\rho}{8}\left(Q^{4}(t_{i})+\frac{1}{|S^{3}|}\int_{\tilde{L}_{i}}|\nabla u_{i}|^{3}\langle y,\nabla u_{i}\rangle d\tilde{l_{i}}\right)\nonumber \\
 & +\frac{1}{|S^{3}|}\int_{\tilde{L}_{i}}\tilde{H}(y)|\nabla u_{i}|^{2}\left(Q(t_{i})-\langle y,\nabla u_{i}\rangle\right)d\tilde{l_{i}}.\nonumber 
\end{align}
We have used the fact that $|\nabla u_{i}|\nu=-\nabla u_{i}$ on the
level set $\tilde{L}_{i}$.  We claim that 
\[
\lim_{i\to\infty}M(t_{i})=0.
\]
Now take the limits of two terms on the right-hand side of (\ref{eq:M(t_i)})
separately.

As $i\to\infty,$ from Theorem \ref{thm:convergence of level set},
$|\nabla u_{i}|\to-\alpha$, $y\cdot\nabla u_{i}(y)\to\alpha$ on
$E$ uniformly. By (\ref{eq:convergence of measure}) and Corollary
\ref{cor: gradient convergence}, as $i\to\infty,$
\begin{align}
Q(t_{i})=-\left(\frac{1}{|S^{3}|}\int_{\tilde{L}_{i}}|\nabla u_{i}(y)|^{3}d\tilde{l}_{i}(y)\right)^{\frac{1}{3}} & \to\alpha,\label{eq: z converges to alpha}
\end{align}
and 
\[
\frac{1}{|S^{3}|}\int_{\tilde{L}_{i}}|\nabla u_{i}|^{3}\langle y,\nabla u_{i}\rangle d\tilde{l}_{i}(y)\to-\alpha^{4}.
\]
In addition, as $i\to\infty,$
\begin{equation}
Q(t_{i})-\langle y,\nabla u_{i}\rangle\to0\quad\text{uniformly\,on}\,\tilde{L}_{i}.\label{eq:Hintegral converge.}
\end{equation}
By (\ref{eq: z converges to alpha}), as $i\to\infty,$
\begin{equation}
Q^{4}(t_{i})+\frac{1}{|S^{3}|}\int_{\tilde{L}_{i}}|\nabla u_{i}|^{3}\langle y,\nabla u_{i}\rangle d\tilde{l}_{i}\to|S^{3}|\left(\alpha^{4}-\alpha^{4}\right)=0.\label{eq:Mass first term goes to 0}
\end{equation}

Next, we claim that
\begin{equation}
\int_{\tilde{L}_{i}}\tilde{H}(y)|\nabla u_{i}|^{2}\left(Q(t_{i})-\langle y,\nabla u_{i}\rangle\right)d\tilde{l}_{i}\to0\quad\text{as}\quad i\rightarrow\infty.\label{eq:Integral of H goes to 0}
\end{equation}
Since by (\ref{eq:Hintegral converge.}),
\[
\left|\int_{\tilde{L}_{i}}\tilde{H}|\nabla u_{i}|^{2}\left(Q(t_{i})-\langle y,\nabla u_{i}\rangle\right)d\tilde{l}_{i}\right|\leq\max_{y\in\tilde{L}_{i}}|Q(t_{i})-\langle y,\nabla u_{i}\rangle|\cdot\int_{\tilde{L}_{i}}|\tilde{H}|\cdot|\nabla u_{i}|^{2}d\tilde{l}_{i},
\]
it is sufficient to prove that $\int_{\tilde{L}_{i}}\tilde{H}|\nabla u_{i}|^{2}d\tilde{l}_{i}$
is uniformly bounded. We notice that $H$, the mean curvature of $L_{t}$,
is positive by (\ref{eq:sigma_1(tilde A) mean curvature is positive}).
By (\ref{eq: z converges to alpha}), $Q(t_{i})$ is uniformly bounded
and 
\begin{align*}
\frac{1}{|S^{3}|}\int_{\tilde{L}_{t_{i}}}|\tilde{H}|\cdot|\nabla u_{i}|^{2}d\tilde{l}_{i} & =\frac{1}{|S^{3}|}\int_{L_{t_{i}}}H|\nabla u|^{2}dl\\
 & =2N(t_{i})+\rho|Q(t_{i})|^{3}\\
 & \leq\frac{1}{|S^{3}|}\int_{\mathbb{R}^{4}}f(u)dx+\rho|Q(t_{i})|^{3}
\end{align*}
is uniformly bounded. We have thus proved (\ref{eq:Integral of H goes to 0}).
By (\ref{eq:M(t_i)}), (\ref{eq:Mass first term goes to 0}) and (\ref{eq:Integral of H goes to 0}),
\begin{equation}
\lim_{i\to\infty}M(t_{i})=0.\label{eq:add61}
\end{equation}

It is clear that $M(\tau')=0$, where $\tau'=\max_{\mathbb{R}^{n}}u$.
By Theorem \ref{thm:mass theorem}, we conclude that $M(t)\equiv0$
and then $u$ is rotationally symmetric. 
\end{proof}

\section{Radial solutions \label{sec:Radial-solutions}}

In this section, we study the radial solution $u(|x|)$ of (\ref{eq:PDE system})
where $f(u)=\frac{3}{2}e^{4u}$.  Denote $r=|x|$. By (\ref{eq:PDE system}),
$u(r)$ satisfies the following differential equation:
\begin{equation}
\left(3u''+\frac{3u'}{r}\right)\left(\frac{u'}{r}+\frac{\rho\left(u'\right)^{2}}{2}\right)=\frac{3}{2}e^{4u},\ u'(0)=0.\label{eq:ode}
\end{equation}
Let $s=\ln r$. We denote $u_{s}:=\frac{du(e^{s})}{ds}=ru'.$ Then,
from (\ref{eq:ode}) we have
\begin{equation}
3u_{ss}\left(u_{s}+\frac{\rho}{2}\left(u_{s}\right)^{2}\right)-\frac{3}{2}e^{4(u+s)}=0,\ u_{s}(-\infty)=0.\label{eq:u_s ode-1}
\end{equation}
The condition that $\lambda(A(\rho,u))\in\Gamma_{2}^{+}$ implies
that 
\begin{equation}
u_{s}+\frac{\rho}{2}\left(u_{s}\right)^{2}<0,u_{ss}<0,\label{eq:positive cone condtion for ode}
\end{equation}
and it also implies $-\frac{2}{\rho}<u_{s}<0$ for $s\in(-\infty,\infty)$.
Note (\ref{eq:u_s ode-1}) is invariant if we rescale $u$ by 
\begin{equation}
\tilde{u}(e^{s})=u(e^{s+c})+c.\label{eq:shift a const}
\end{equation}

\begin{lem}
(\ref{eq:u_s ode-1}) admits a first integral: 
\begin{equation}
(u_{s})^{2}\left(\frac{\rho}{4}u_{s}^{2}+\frac{2+\rho}{3}u_{s}+1\right)-\frac{1}{4}e^{4(u+s)}\equiv0.\label{eq:first inegral 0}
\end{equation}
\end{lem}

\begin{proof}
Multiply $(u_{s}+1)$ to both sides of (\ref{eq:u_s ode-1}) and get
\begin{align*}
0 & =3u_{ss}(u_{s}+1)\left(u_{s}+\frac{\rho}{2}\left(u_{s}\right)^{2}\right)-\frac{3}{2}e^{4(u+s)}\left(u_{s}+1\right)\\
 & =\frac{3}{2}\left[(u_{s})^{2}\left(\frac{\rho}{4}u_{s}^{2}+\frac{2+\rho}{3}u_{s}+1\right)-\frac{1}{4}e^{4(u+s)}\right]_{s}.
\end{align*}
Thus, we obtain a first integral of (\ref{eq:u_s ode-1}):
\begin{equation}
(u_{s})^{2}\left(\frac{\rho}{4}u_{s}^{2}+\frac{2+\rho}{3}u_{s}+1\right)-\frac{1}{4}e^{4(u+s)}\equiv c_{0},\label{eq:first integral}
\end{equation}
for some constant $c_{0}$. As $s\to-\infty,$ we have $\lim_{s\to-\infty}u_{s}=\lim_{r\to0}ru'=0$.
Hence, $c_{0}=0$ and 
\begin{equation}
(u_{s})^{2}\left(\frac{\rho}{4}u_{s}^{2}+\frac{2+\rho}{3}u_{s}+1\right)-\frac{1}{4}e^{4(u+s)}=0.\label{eq:first integral.}
\end{equation}
\end{proof}

\begin{thm}
For $\rho\in[0,2)$, (\ref{eq:ode}) admits a unique radial solution
up to a rescaling. 
\end{thm}

\begin{proof}
To fix the gauge, we may assume that $u_{s}|_{s=0}=-\epsilon$ for
some $\epsilon>0$ after a rescaling in (\ref{eq:shift a const}),
and consider
\begin{equation}
3u_{ss}\left(u_{s}+\frac{\rho}{2}\left(u_{s}\right)^{2}\right)-\frac{3}{2}e^{4(u+s)}=0,\ u_{s}(-\infty)=0,\ u_{s}|_{s=0}=-\epsilon.\label{eq:u_s ode}
\end{equation}
By (\ref{eq:first inegral 0}) and (\ref{eq:u_s ode})
\[
(u_{s})^{2}\left(\frac{\rho}{4}u_{s}^{2}+\frac{2+\rho}{3}u_{s}+1\right)=\frac{1}{2}u_{ss}\left(u_{s}+\frac{\rho}{2}\left(u_{s}\right)^{2}\right).
\]
Then
\begin{equation}
\int_{-\epsilon}^{u_{s}}\frac{(1+\frac{\rho}{2}x)}{2x\left(\frac{\rho}{4}x^{2}+\frac{2+\rho}{3}x+1\right)}dx=\int_{0}^{s}ds.\label{eq:u_s integral}
\end{equation}

Case 1. If $\rho=0$, then 
\[
\int_{-\epsilon}^{u_{s}}\frac{1}{2x\left(\frac{2}{3}x+1\right)}dx=\int_{0}^{s}ds.
\]
Thus, 
\[
\frac{1}{2}\ln\left(\left|\frac{u_{s}}{3+2u_{s}}\right|\right)-\frac{1}{2}\ln\left(\frac{\epsilon}{3-2\epsilon}\right)=s.
\]
We obtain that 
\[
u_{s}=-\frac{\frac{3}{2}e^{2s}}{e^{2s}+\frac{1}{2}\left(\frac{3-2\epsilon}{\epsilon}\right)},\ u=-\frac{3}{4}\ln\left(e^{2s}+\frac{1}{2}\left(\frac{3-2\epsilon}{\epsilon}\right)\right).
\]
Since $u\sim-\frac{3}{2}\ln r$ as $r\to\infty$, the total integral
$\int_{\mathbb{R}^{4}}e^{4u}<\infty$. 

Case 2. Suppose $\rho\in(0,2)$. Now, $x\left(\frac{\rho}{4}x^{2}+\frac{2+\rho}{3}x+1\right)=0$
has 3 roots 
\[
x_{0}=0,\ x_{1}=-\frac{\left(\rho+2\right)-\sqrt{\rho^{2}-5\rho+4}}{3\rho/4},\ x_{2}=-\frac{\left(\rho+2\right)+\sqrt{\rho^{2}-5\rho+4}}{3\rho/4}.
\]
Simple calculation shows $-1>x_{1}>-\frac{2}{\rho}$. Thus, $\frac{(1+\frac{\rho}{2}x)}{2x\left(\frac{\rho}{4}x^{2}+\frac{2+\rho}{3}x+1\right)}<0$
for $x\in(x_{1},0)$, and $u_{s}$ defined by (\ref{eq:u_s integral})
is decreasing in $s$. From (\ref{eq:u_s integral}), $u_{s}\to x_{0}=0$,
as $s\to-\infty$, and $u_{s}\to x_{1}$, as $s\to\infty$. From
the first integral (\ref{eq:first inegral 0}), the solution is uniquely
determined. Since $\lim_{s\to\infty}u_{s}=x_{1}<-1$, the total integral
\[
\int_{\mathbb{R}^{4}}e^{4u(|x|)}dx=|S^{3}|\int_{-\infty}^{\infty}e^{4u(e^{s})}e^{4s}ds<\infty.
\]
\end{proof}
\begin{rem}
The existence of radial solutions to (\ref{eq:PDE system}) for general
$f(u)=e^{4u}p(u)$ cannot be achieved through the previous argument.
It will be an interesting question to find such solutions and see
if the polynomial term $p(u)$ affects the asymptotic behavior. We
leave the question for future studies.
\end{rem}

\appendix

\section{Pohozaev identity}

In this appendix, we prove a Pohozaev identity, which is stated for
a slightly more general form of PDEs. For $k$-Hessian equation, Tso
in \cite{tso1990remarks} and Brandolini-Nitsch-Salani-Trombetti \cite{brandolini2008serrin}
have given the Pohozaev identity by the variational method of Pucci-Serrin
\cite{Pucci1986AGV}. Schoen in \cite{schoen1988existence} mentioned
that Kazdan-Warner identity is equivalent to Pohozaev identity. Deducing
Pohozaev identity from Kazdan-Warner identity is not obvious and Gover-$\textrm{Ø}$rsted
in \cite{gover2013universal} provided a proof for this fact. For
$\sigma_{k}$-Yamabe equation, the Kazdan-Warner identity was established
by Han \cite{Han06} and Viaclovsky \cite{viaclovsky2000some} respectively.
We will adopt the argument of \cite{Pucci1986AGV} to deduce the Pohozaev
identity for the more general form.

For a fixed constant $\tau$, consider the following Dirichlet problem
on a bounded open domain $\Omega$ : 
\begin{equation}
\begin{cases}
\sigma_{2}(A(\rho,u))=K(x)f(u), & \Omega\subset\mathbb{R}^{4},\\
A(\rho,u)\in\Gamma_{2}^{+}, & \ \\
u=\tau, & \partial\Omega.
\end{cases}\label{eq:K equation}
\end{equation}
We assume that $\Omega$ has finite perimeter so that we can use generalized
divergence theorem ( \cite{evans1991measure}, Section 5.8, Theorem
1). Note that $\Omega=\{u>\tau\}.$ For simplicity, we fix a function
$F(t)$ such that $F(\tau)=0$ and $F'(t)=f(t).$ We define a vector
field $X=\sum x^{i}\frac{\partial}{\partial x^{i}}.$ We first prove
that
\begin{lem}
\label{lem:P1}Let $u\in C^{2}(\overline{\Omega})$ be a solution
to (\ref{eq:K equation}) and use notations as above. Then the following
identity holds
\end{lem}

\[
\int_{\Omega}X(u)\cdot\sigma_{2}(A(\rho,u))dx=-\int_{\Omega}(4K(x)+X(K(x)))F(u)dx.
\]

\begin{proof}
Note that, since $F'=f$ and (\ref{eq:K equation}),
\begin{align}
{\rm div}(F(u(x))K(x)X) & =({\rm div}X)F(u(x))K(x)+X(F(u(x))K(x))\nonumber \\
 & =({\rm div}X)F(u(x))K(x)+f(u)X(u(x))K(x)+F(u)X(K(x)).\label{01}
\end{align}
We integrate (\ref{01}). Using the divergence theorem and the fact
that $F(u)|_{\partial\Omega}=0$, we have proved the lemma. 
\end{proof}
Let 
\[
T_{ij}^{1}:=\left(-\Delta u\delta_{ij}+u_{ij}\right),T_{ij}^{2}:=u_{ik}T_{kj}^{1}+\sigma_{2}(-\nabla^{2}u)\delta_{ij}
\]
be the Newton tensors of $\sigma_{2}(-\nabla^{2}u)$ and $\sigma_{3}(-\nabla^{2}u)$.
It is well known that $-T_{ij}^{1}u_{ij}=2\sigma_{2}(-\nabla^{2}u).$
Also, it is easy to verify that if $u$ is in $C^{3}(\Omega)$, then
\begin{equation}
\frac{\partial}{\partial x^{j}}T_{ij}^{1}=0,\ \frac{\partial}{\partial x^{j}}T_{ij}^{2}=0.\label{eq:divergence free}
\end{equation}

We claim the following: 
\begin{lem}
\label{lem:P2}Let $u\in C^{3}(\Omega)\cap C^{2}(\overline{\Omega})$.
Then, 
\begin{align*}
-2\int_{\Omega}\langle x,\nabla u\rangle\sigma_{2}(A(\rho,u))dx\\
=-\frac{2}{3}\int_{\partial\Omega} & \left(-x^{i}u_{i}T_{kj}^{1}u_{j}+u_{j}T_{jk}^{1}(u-\tau)+x^{i}T_{ik}^{2}(u-\tau)\right)\nu_{k}dl.
\end{align*}
\end{lem}

\begin{proof}
Similar computations have been carried out for the 2-Hessian equations
and the $\sigma_{2}$-Yamabe equation in conformal geometry, which
we follow closely. 

We claim that 
\begin{equation}
\int_{\Omega}\langle x,\nabla u\rangle\frac{\partial}{\partial x^{i}}\left(|\nabla u|^{2}u_{i}\right)=\frac{3}{4}\int_{\partial\Omega}|\nabla u|^{4}\langle x,\nu\rangle dl,\label{eq:4laplacian pohozaev}
\end{equation}
and 
\begin{equation}
\int_{\Omega}\langle x,\nabla u\rangle\sigma_{2}(\nabla^{2}u)dx=\frac{1}{3}\int_{\partial\Omega}x^{i}u_{i}H|\nabla u|^{2}dl.\label{eq:sigma hessian pohozaev}
\end{equation}
Since $2\sigma_{2}(A(\rho,u))=2\sigma_{2}(-\nabla^{2}u)+\rho\frac{\partial}{\partial x^{i}}(|\nabla u|^{2}u_{i})$,
(\ref{eq:k-z formula}) follows immediately. 

To prove (\ref{eq:4laplacian pohozaev}), we compute directly. 
\begin{align}
x^{j}u_{j}\frac{\partial\left(|\nabla u|^{2}u_{i}\right)}{\partial x^{i}} & =\frac{\partial\left(x^{j}u_{j}|\nabla u|^{2}u_{i}\right)}{\partial x^{i}}-\left(u_{i}+x^{j}u_{ji}\right)|\nabla u|^{2}u_{i}\label{eq:p-laplacian}\\
 & =\frac{\partial\left(x^{j}u_{j}|\nabla u|^{2}u_{i}\right)}{\partial x^{i}}-\left(|\nabla u|^{4}+\frac{1}{4}x^{j}\frac{\partial}{\partial x^{j}}|\nabla u|^{4}\right)\nonumber \\
 & =\frac{\partial\left(x^{j}u_{j}|\nabla u|^{2}u_{i}\right)}{\partial x^{i}}-\left(|\nabla u|^{4}+\frac{1}{4}\frac{\partial\left(x^{j}|\nabla u|^{4}\right)}{\partial x^{j}}-|\nabla u|^{4}\right)\nonumber \\
 & =\frac{\partial}{\partial x^{i}}\left(x^{j}u_{j}|\nabla u|^{2}u_{i}-\frac{1}{4}x^{i}|\nabla u|^{4}\right).\nonumber 
\end{align}
By (\ref{eq:p-laplacian}), (\ref{eq:nu def}), and the divergence
theorem,
\[
\int_{\Omega}x^{j}u_{j}\frac{\partial\left(|\nabla u|^{2}u_{i}\right)}{\partial x^{i}}dx=\frac{3}{4}\int_{\partial\Omega}|\nabla u|^{4}x^{i}\nu_{i}dl.
\]
Thus, we have proved the first claim. 

Next, we prove (\ref{eq:sigma hessian pohozaev}). Recall in (\ref{eq:divergence free})
that tensors $T^{1}$ and $T^{2}$ are divergence free. Also, $-T_{ij}^{1}u_{ij}=2\sigma_{2}(-\nabla^{2}u)$.
From (\ref{eq:divergence free}), $2\sigma_{2}(-\nabla^{2}u)=-\frac{\partial}{\partial x^{i}}\left(T_{ij}^{1}u_{j}\right)$.
We have 
\begin{equation}
2\int_{\Omega}x^{i}u_{i}\sigma_{2}(-\nabla^{2}u)dx=-\int_{\Omega}\frac{\partial}{\partial x^{j}}\left(x^{i}u_{i}T_{jk}^{1}u_{k}\right)dx+\int_{\Omega}\frac{\partial}{\partial x^{j}}\left(x^{i}u_{i}\right)T_{jk}^{1}u_{k}dx.\label{eq:intermediate -1}
\end{equation}
The following holds in $\Omega$:
\begin{align}
\frac{\partial\left(x^{i}u_{i}\right)}{\partial x^{j}}T_{jk}^{1}u_{k}dx & =u_{j}T_{jk}^{1}u_{k}+\left(x^{i}u_{ij}\right)T_{jk}^{1}u_{k}\label{eq:intermediate0}\\
 & =\frac{\partial}{\partial x^{k}}\left(u_{j}T_{jk}^{1}(u-\tau)\right)-u_{jk}T_{jk}^{1}(u-\tau)+x^{i}u_{ij}T_{jk}^{1}u_{k}\nonumber \\
 & =\frac{\partial}{\partial x^{k}}\left(u_{j}T_{jk}^{1}(u-\tau)\right)+2\sigma_{2}(-\nabla^{2}u)(u-\tau)+x^{i}u_{ij}T_{jk}^{1}u_{k};\nonumber 
\end{align}
 
\begin{align}
x^{i}u_{ij}T_{jk}^{1}u_{k} & =x^{i}\left(T_{ik}^{2}-\sigma_{2}(-\nabla^{2}u)\delta_{ik}\right)u_{k}\label{eq:intermediate}\\
 & =\frac{\partial}{\partial x^{k}}\left(x^{i}T_{ik}^{2}(u-\tau)\right)-\delta_{k}^{i}T_{ik}^{2}(u-\tau)-x^{k}u_{k}\sigma_{2}(-\nabla^{2}u)\nonumber \\
 & =\frac{\partial}{\partial x^{k}}\left(x^{i}T_{ik}^{2}(u-\tau)\right)-2\sigma_{2}(-\nabla^{2}u)(u-\tau)-x^{k}u_{k}\sigma_{2}(-\nabla^{2}u).\nonumber 
\end{align}
We have used the fact that $T_{ii}^{2}=2\sigma_{2}(-\nabla^{2}u)$
in the last equality. By (\ref{eq:intermediate -1}),(\ref{eq:intermediate0}),
(\ref{eq:intermediate}), and the divergence theorem, we have 
\begin{align}
3\int_{\Omega}x^{i}u_{i}\sigma_{2}(-\nabla^{2}u)dx & =\int_{\Omega}\left(\frac{\partial}{\partial x^{k}}\left(-x^{i}u_{i}T_{kj}^{1}u_{j}+u_{j}T_{jk}^{1}(u-\tau)+x^{i}T_{ik}^{2}(u-\tau)\right)\right)dx\label{eq:intermediate 2}\\
 & =\int_{\partial\Omega}\left(-x^{i}u_{i}T_{kj}^{1}u_{j}+u_{j}T_{jk}^{1}(u-\tau)+x^{i}T_{ik}^{2}(u-\tau)\right)\nu_{k}dl.\nonumber 
\end{align}
\end{proof}
\begin{lem}
\label{lem:P3}Assume that $\Omega$ has almost $C^{2}$ boundary,
i.e. the singular part of $\partial\Omega$ has zero $\mathcal{H}^{3}$-measure.
If $u\in C^{3}(\Omega)\cap C^{2}(\bar{\Omega})$ and $u|_{\partial\Omega}=\tau$,
then we have 
\begin{equation}
-2\int_{\Omega}\langle x,\nabla u\rangle\sigma_{2}(A(\rho,u))dx=\int_{\partial\Omega}\left(-\frac{3}{4}\rho|\nabla u|^{4}\langle x,\nu\rangle+\frac{2}{3}H|\nabla u|^{3}\langle x,\nu\rangle\right)dl.\label{eq:k-z formula}
\end{equation}
\end{lem}

\begin{proof}
As $\Omega$ has almost $C^{2}$ boundary, the regular part of $\partial\Omega$
is a $C^{2}$ hypersurface. The singular part has zero $\mathcal{H}^{3}$-measure.
As $u\in C^{2}(\overline{\Omega})$ and Lemma \ref{lem:P2}, the integral
on the singular part of $\partial\Omega$ is zero. Using $u|_{\partial\Omega}=\tau$,
and the fact that 
\[
u_{j}T_{jk}^{1}\nu_{k}|_{M}=\left(-u_{\nu}\Delta u+u_{\nu\nu}u_{\nu}\right)=H|\nabla u|u_{\nu}=\left(-H|\nabla u|^{2}\right)|_{M},
\]
we conclude from (\ref{eq:intermediate 2}) that 
\begin{align*}
\int_{\Omega}x^{i}u_{i}\sigma_{2}(-\nabla^{2}u)dx & =\frac{1}{3}\int_{\partial\Omega}x^{i}u_{i}T_{jk}^{1}u_{j}\nu_{k}dl\\
 & =\frac{1}{3}\int_{\partial\Omega}x^{i}u_{i}H|\nabla u|^{2}dl.
\end{align*}
\end{proof}

\begin{lem}
\label{lem:P4-boundary smooth}Assume that $\partial\Omega$ is smooth.
If $u\in C^{2}(\overline{\Omega})$ and $u|_{\partial\Omega}=\tau$,
then we have 
\begin{equation}
-2\int_{\Omega}\langle x,\nabla u\rangle\sigma_{2}(A(\rho,u))dx=\int_{\partial\Omega}\left(-\frac{3}{4}\rho|\nabla u|^{4}\langle x,\nu\rangle+\frac{2}{3}H|\nabla u|^{3}\langle x,\nu\rangle\right)dl.\label{eq:k-z formula-1}
\end{equation}
\end{lem}

\begin{proof}
Since $\partial\Omega$ is smooth, we may extend $u$ to be a $C^{2}$
function on $\mathbb{R}^{4}$. Then, we can find a sequence of $C^{3}$
functions $u_{p}$ such that $||u_{p}-u||_{C^{2}(\overline{\Omega})}\rightarrow0$
for $p\rightarrow\infty$. From Lemma \ref{lem:P2}, we know that
\begin{align*}
2\int_{\Omega}\langle x,\nabla u_{p}\rangle\sigma_{2}(A(\rho,u_{p}))dx & =\\
\frac{2}{3}\int_{\partial\Omega}\nu_{k} & \left(-x^{i}u_{p,i}T_{kj}^{1}u_{p,j}+u_{p,j}T_{jk}^{1}(u_{p}-\tau)+x^{i}T_{ik}^{2}(u_{p}-\tau)\right)dl,
\end{align*}
where $T_{ij}^{1}=\left(-\Delta u_{p}\delta_{ij}+u_{p,ij}\right)$,
and $T_{ij}^{2}=u_{p,ik}T_{kj}^{1}+\sigma_{2}(-\nabla^{2}u_{p})\delta_{ij}$.

By $||u_{p}-u||_{C^{2}(\overline{\Omega})}\rightarrow0$ for $p\rightarrow\infty$
and $u=\tau|_{\partial\Omega}$, we have 
\[
-2\int_{\Omega}\langle x,\nabla u\rangle\sigma_{2}(A(\rho,u))dx=\frac{2}{3}\int_{\partial\Omega}x^{i}u_{i}T_{kj}^{1}u_{j}\nu_{k}dl.
\]
Similar to the argument in Lemma \ref{lem:P3}, we have proved this
Lemma.
\end{proof}
Combining Lemmas \ref{lem:P1} , \ref{lem:P3}, and \ref{lem:P4-boundary smooth},
we have the following Pohozaev type identity:
\begin{thm}
\label{thm:pohozaev}Suppose either of the following conditions holds.

$\text{1}.$ $\Omega$ has almost $C^{2}$ boundary, and $u$ is a
$C^{3}(\Omega)\cap C^{2}(\overline{\Omega})$ solution to (\ref{eq:K equation});

$2$. $\Omega$ has smooth boundary, and $u$ is a $C^{2}(\overline{\Omega})$
solution to (\ref{eq:K equation}).

Then, we have 

\begin{equation}
\int_{\Omega}8(K(x)+\frac{1}{4}\langle x,\nabla K\rangle)F(u)dx=\int_{\partial\Omega}\left(-\frac{3}{4}\rho|\nabla u|^{4}\langle x,\nu\rangle+\frac{2}{3}H|\nabla u|^{3}\langle x,\nu\rangle\right)dl,\label{eq:Pohozaev identity for general f}
\end{equation}
where $\nu$ is the unit outward normal vector of $\partial\Omega$. 
\end{thm}

Theorem \ref{thm:pohozaev} can also be established using a classical
variational identity of Pucci-Serrin \cite{Pucci1986AGV}. We present
this approach for comparison.

We set up some notations first. Let 
\[
\mathcal{F}:\mathbb{R}^{n}\times\mathbb{R}\times\mathbb{R}^{n}\times Sym(\mathbb{R}^{n\times n})\to\mathbb{R}
\]
\[
(x,u,\mathbf{p},\mathbf{r})\mapsto\mathcal{F}(x,u,\mathbf{p},\mathbf{r})
\]
be a smooth function where $\mathbf{p}=(p_{1},\cdots,p_{n})\in\mathbb{R}^{n}$,
$\mathbf{r}=(r_{ij})_{n\times n}$. Consider the variational problem
on an open domain $\Omega^{n}\subset\mathbb{R}^{n}$
\begin{equation}
\delta\int_{\Omega^{n}}\mathcal{F}(x,u,\nabla u,\nabla^{2}u)dx=0.\label{eq:variational problem}
\end{equation}
Denote $\mathcal{F}_{u}=\frac{\partial\mathcal{F}}{\partial u}$,
$\mathcal{F}_{p_{i}}=\frac{\partial\mathcal{F}}{\partial p_{i}}$,
$\mathcal{F}_{r_{ij}}=\frac{\partial\mathcal{F}}{\partial r_{ij}}$,
and $\mathcal{F}_{x^{i}}=\frac{\partial\mathcal{F}}{\partial x^{i}}$.
The corresponding Euler-Lagrange equation for (\ref{eq:variational problem})
is 
\begin{equation}
\frac{\partial^{2}\mathcal{F}_{r_{ij}}}{\partial x^{i}\partial x^{j}}-\frac{\partial\mathcal{F}_{p_{i}}}{\partial x^{i}}+\mathcal{F}_{u}=0.\label{eq:euler-lag}
\end{equation}
Then, the following identity holds.
\begin{thm}
[\cite{Pucci1986AGV}, Proposition 3]\label{thm:generalized form-PS}Let
$u\in C^{4}(\Omega)$ be a solution of (\ref{eq:euler-lag}). Then
\begin{equation}
\ensuremath{\begin{aligned}\frac{\partial}{\partial x^{i}}\bigg[x^{i}\mathcal{F}- & x^{k}\frac{\partial u}{\partial x^{k}}\big(\mathcal{F}_{p_{i}}-\frac{\partial}{\partial x_{j}}\mathcal{F}_{r_{ij}}\big)-\mathcal{F}_{r_{ij}}\frac{\partial}{\partial x^{j}}\big(x^{k}\frac{\partial u}{\partial x^{k}}\big)\bigg]\\
 & =n\mathcal{F}+x^{i}\mathcal{F}_{x^{i}}-\frac{\partial u}{\partial x^{i}}\mathcal{F}_{p_{i}}-2\frac{\partial^{2}u}{\partial x^{i}\partial x^{j}}\mathcal{F}_{r_{ij}}.
\end{aligned}
}\label{eq:PS-identity}
\end{equation}
\end{thm}

For (\ref{eq:K equation}), we pick the following
\begin{equation}
\mathcal{F}(x,u,\nabla u,\nabla^{2}u)=-2\frac{(u-\tau)\sigma_{2}(\nabla^{2}u)}{3}+\frac{\rho}{4}|\nabla u|^{4}+2K(x)F(u).\label{eq:our functional mathcal F}
\end{equation}
The partial derivatives of $\mathcal{F}$ are given by
\begin{equation}
\mathcal{F}_{u}=-\frac{2\sigma_{2}(\nabla^{2}u)}{3}+2K(x)f(u),\ \mathcal{F}_{x^{i}}=2\langle x,\nabla K\rangle F(u),\label{eq:partial of functional F}
\end{equation}
\begin{equation}
\mathcal{F}_{p_{i}}=\rho|\nabla u|^{2}u_{i},\quad\mathcal{F}_{r_{ij}}=-\frac{2}{3}(u-\tau)\frac{\partial\sigma_{2}}{\partial r_{ij}}=-\frac{2}{3}(u-\tau)\left(\delta_{ij}\Delta u-u_{ij}\right).\label{eq:partial derivative of functional F}
\end{equation}

\begin{rem}
\label{rem:u in C^3}Notice that (\ref{eq:divergence free}) holds
if we assume $u\in C^{3}(\Omega)$. Thus,
\begin{align*}
\frac{\partial}{\partial x^{j}}\mathcal{F}_{r_{ij}} & =-\frac{2}{3}u_{j}\left(\delta_{ij}\Delta u-u_{ij}\right)
\end{align*}
which only needs derivatives up to 2nd order. Similarly, $\frac{\partial^{2}}{\partial x^{j}\partial x^{i}}\mathcal{F}_{r_{ij}}$
in (\ref{eq:euler-lag}) just needs derivatives up to order 2. Hence,
for our functional (\ref{eq:our functional mathcal F}), Theorem \ref{thm:generalized form-PS}
holds for $u\in C^{3}(\Omega)\cap C^{2}(\bar{\Omega})$. 
\end{rem}

The Euler-Lagrangian equation (\ref{eq:euler-lag} ) in $\Omega$
is 
\[
\sigma_{2}(A(\rho,u))-K(x)f(u)=0.
\]
See \cite{brendle2004variational} the variational function for $\sigma_{\frac{n}{2}}(A(1,u))$
on closed manifolds.

Now, we present the proof of Theorem \ref{thm:pohozaev} using (\ref{eq:PS-identity}).
\begin{proof}
[Another proof of Theorem \ref{thm:pohozaev}]We first assume that
$u$ is a solution in $C^{3}(\Omega)\cap C^{2}(\bar{\Omega})$ on
a domain $\Omega$ with almost $C^{2}$ boundary. By (\ref{eq:partial of functional F})
and (\ref{eq:partial derivative of functional F}) ,

\[
4\mathcal{F}+x^{i}\mathcal{F}_{x^{i}}-\frac{\partial u}{\partial x^{i}}\mathcal{F}_{p_{i}}-2\frac{\partial^{2}u}{\partial x^{i}\partial x^{j}}\mathcal{F}_{r_{ij}}=8F(u)K(x)+2\langle x,\nabla K\rangle F(u).
\]
By Remark \ref{rem:u in C^3}, we can apply (\ref{eq:PS-identity}).
Integrate both sides of (\ref{eq:PS-identity}) to get
\begin{align}
\int_{\partial\Omega}\left(\mathcal{F}\langle x,\nu\rangle-\langle x,\nabla u\rangle\left(\mathcal{F}_{p_{i}}-\frac{\partial}{\partial x^{j}}\mathcal{F}_{r_{ij}}\right)\nu_{i}-\frac{\partial\langle x,\nabla u\rangle}{\partial x^{j}}\mathcal{F}_{r_{ij}}\nu_{i}\right)dl\label{eq:PS1}\\
=\int_{\Omega}8(K(x)+\frac{1}{4}\langle x,\nabla K\rangle)F(u)dx.\nonumber 
\end{align}
We simplify the left hand side of (\ref{eq:PS1}) by facts that $u|_{\partial\Omega}\equiv\tau$,
$\mathcal{F}|_{\partial\Omega}=\frac{\rho}{4}|\nabla u|^{4}$, $\nu=-\frac{\nabla u}{|\nabla u|}$,
and (\ref{eq:new100}) to get

\begin{align*}
\langle\mathcal{F}_{p_{i}}-\frac{\partial}{\partial x^{j}}\mathcal{F}_{r_{ij}},\nu_{i}\rangle & =\langle\frac{\rho}{4}4|\nabla u|^{2}u_{i}+\frac{2}{3}u_{j}\left(\delta_{ij}\Delta u-u_{ij}\right),\nu_{i}\rangle\\
 & =-\rho|\nabla u|^{3}+\frac{2}{3}H|\nabla u|^{2}.
\end{align*}
We have thus proved our theorem for $u\in C^{3}(\Omega)\cap C^{2}(\bar{\Omega}).$

If the boundary of $\Omega$ is smooth, we may first extend $u$ to
be a $C^{2}$ function on $\mathbb{R}^{4}$. We can find a sequence
of $C^{3}$ functions $u_{p}$ such that $||u_{p}-u||_{C^{2}(\overline{\Omega})}\rightarrow0$.
By $C^{2}$ convergence, (\ref{eq:Pohozaev identity for general f})
holds for $u$. We finish the proof.
\end{proof}
\bibliographystyle{plain}
\addcontentsline{toc}{section}{\refname}\bibliography{BibforLiouville}

\end{document}